\documentclass[final,onefignum,onetabnum]{siamart220329}

\usepackage{braket,amsfonts}

\usepackage{array}

\usepackage[caption=false]{subfig}
\usepackage{pgfplots}

\newsiamthm{claim}{Claim}
\newsiamremark{remark}{Remark}
\newsiamremark{hypothesis}{Hypothesis}
\crefname{hypothesis}{Hypothesis}{Hypotheses}

\usepackage{algorithmic}

\usepackage{graphicx,epstopdf}

\Crefname{ALC@unique}{Line}{Lines}

\usepackage{amsopn}

\usepackage{xspace}
\usepackage{bold-extra}
\usepackage[most]{tcolorbox}

\colorlet{texcscolor}{blue!50!black}
\colorlet{texemcolor}{red!70!black}
\colorlet{texpreamble}{red!70!black}
\colorlet{codebackground}{black!25!white!25}

\lstdefinestyle{siamlatex}{%
  style=tcblatex,
  texcsstyle=*\color{texcscolor},
  texcsstyle=[2]\color{texemcolor},
  keywordstyle=[2]\color{texemcolor},
  moretexcs={cref,Cref,maketitle,mathcal,text,headers,email,url},
}

\tcbset{%
  colframe=black!75!white!75,
  coltitle=white,
  colback=codebackground, % bottom/left side
  colbacklower=white, % top/right side
  fonttitle=\bfseries,
  arc=0pt,outer arc=0pt,
  top=1pt,bottom=1pt,left=1mm,right=1mm,middle=1mm,boxsep=1mm,
  leftrule=0.3mm,rightrule=0.3mm,toprule=0.3mm,bottomrule=0.3mm,
  listing options={style=siamlatex}
}

\newtcblisting[use counter=example]{example}[2][]{%
  title={Example~\thetcbcounter: #2},#1}

\newtcbinputlisting[use counter=example]{\examplefile}[3][]{%
  title={Example~\thetcbcounter: #2},listing file={#3},#1}

\DeclareTotalTCBox{\code}{ v O{} }
{ %fontupper=\ttfamily\color{texemcolor},
  fontupper=\ttfamily\color{black},
  nobeforeafter,
  tcbox raise base,
  colback=codebackground,colframe=white,
  top=0pt,bottom=0pt,left=0mm,right=0mm,
  leftrule=0pt,rightrule=0pt,toprule=0mm,bottomrule=0mm,
  boxsep=0.5mm,
  #2}{#1}

\patchcmd\newpage{\vfil}{}{}{}
\begin{tcbverbatimwrite}{tmp_\jobname_header.tex}
\title{ two-person zero-sum Stochastic linear quadratic control problems with Markov chains and fractional Brownian motion  in infinite horizon}

\author{Chang Liu, Hongtao Fan,
 Yajing Li\thanks{College of Science, Northwest A\&F University, Yangling, Shaanxi 712100, People’s Republic of China}}

\headers{Two-person zero-sum stochastic control on infinite time}{C. Liu, H. Fan, and Y. Li}
\end{tcbverbatimwrite}
\title{ two-person zero-sum Stochastic linear quadratic control problems with Markov chains and fractional Brownian motion  in infinite horizon}

\author{Chang Liu, Hongtao Fan,
 Yajing Li\thanks{College of Science, Northwest A\&F University, Yangling, Shaanxi 712100, People’s Republic of China}}

\headers{Two-person zero-sum stochastic control on infinite time}{C. Liu, H. Fan, and Y. Li}

\ifpdf
\hypersetup{ pdftitle={Guide to Using  SIAM'S \LaTeX\ Style} }
\fi

\begin{document}
\maketitle
\bibliographystyle{siamplain}
%% ------------------------------------------------------------------
%% ABSTRACT
%%
%\bibliography{references} 

\begin{tcbverbatimwrite}{tmp_\jobname_abstract.tex}
\begin{abstract}
 This paper addresses a class of two-person zero-sum stochastic differential equations, which encompass Markov chains and fractional Brownian motion, and satisfy some monotonicity conditions over an infinite time horizon. Within the framework of forward-backward stochastic differential equations (FBSDEs) that describe system evolution, we extend the classical It$\rm\hat{o}$'s formula to accommodate complex scenarios involving Brownian motion, fractional Brownian motion, and Markov chains simultaneously. By applying the Banach fixed-point theorem and approximation methods respectively, we theoretically guarantee the existence and uniqueness of solutions for FBSDEs in infinite horizon. Furthermore, we apply the method for the first time to the optimal control problem in a two-player zero-sum game, deriving the optimal control strategies for both players by solving the FBSDEs system. Finally, we conduct an analysis of the impact of the cross-term 
 $S(\cdot)$  in the cost function on the solution, revealing its crucial role in the optimization process. 
\end{abstract}


%\begin{MSCcodes}
%00A20, 00B10
%\end{MSCcodes}
\end{tcbverbatimwrite}
\begin{abstract}
 This paper addresses a class of two-person zero-sum stochastic differential equations, which encompass Markov chains and fractional Brownian motion, and satisfy some monotonicity conditions over an infinite time horizon. Within the framework of forward-backward stochastic differential equations (FBSDEs) that describe system evolution, we extend the classical It$\rm\hat{o}$'s formula to accommodate complex scenarios involving Brownian motion, fractional Brownian motion, and Markov chains simultaneously. By applying the Banach fixed-point theorem and approximation methods respectively, we theoretically guarantee the existence and uniqueness of solutions for FBSDEs in infinite horizon. Furthermore, we apply the method for the first time to the optimal control problem in a two-player zero-sum game, deriving the optimal control strategies for both players by solving the FBSDEs system. Finally, we conduct an analysis of the impact of the cross-term
 $S(\cdot)$  in the cost function on the solution, revealing its crucial role in the optimization process.
\end{abstract}


%\begin{MSCcodes}
%00A20, 00B10
%\end{MSCcodes}

%% ------------------------------------------------------------------

%% END HEADER
%% ------------------------------------------------------------------

\section{Introduction}
Optimal control theory is a mathematical framework used to determine how to optimize the performance of a system by controlling its variables subject to given constraints. It has broad applications across various fields, including engineering, economics, and physics. The linear-quadratic (LQ) optimal control problem \cite{yong2012stochastic}, which features a quadratic objective function and linear constraints, is particularly useful for modeling systems where the goal is to minimize key performance measures such as energy consumption, time costs, or system errors.

A significant body of work has focused on finite-time stochastic control problems involving factors such as mean-field interactions, time delay, partial information, open/closed-loop systems, and Markov chains. Research has explored both open-loop and closed-loop solvability of the cost function using the Riccati equation \cite{sun2016open}.   The optimal control problem for systems with Markov state switching in finite time has been investigated in \cite{wen2023stochastic}. A class of mean-field LQ problems under monotonicity conditions is addressed in \cite{tian2023mean}. In \cite{li2020linear}, optimal control problems with time delay and full and partial information are considered. Research on infinite-time stochastic optimal control has also emerged, including investigations on forward-backward stochastic differential equations (FBSDEs) with jump diffusion \cite{yu2017infinite}, and the exploration of open-loop solutions for FBSDEs with random coefficients \cite{wei2021infinite}. However, while methods for solving relatively simpler stochastic differential equations in infinite time have been developed, there remains a notable gap in the study of more complex models, especially those incorporating fractional Brownian motion (fBm) and higher-order interactions.

Most existing methods in this field focus on stochastic differential equations (SDEs) driven by Brownian motion, which is commonly employed to model noise or disturbances in systems. However, in many practical scenarios, physical phenomena cannot be adequately explained by Brownian motion alone.  Since fractional Brownian motion is neither a Markov process nor a semimartingale, it can provide a better explanation for certain complex phenomena. For instance, in \cite{han2019stochastic}, the stochastic LQ problem driven by fBm is analyzed. Similarly, the optimal control problem for linear stochastic PDEs driven by multiplicative fBm is investigated in \cite{grecksch2022optimal}. A class of SDEs in Hilbert space driven by fBm is explored in \cite{duncan2012linear}. Nevertheless, the research in these researches is constrained to finite time horizons.

The study of multi-player control problems in the field of stochastic control is a highly complex and challenging area of research. These problems typically involve multiple decision-makers, such as controllers or participants who must navigate an uncertain environment to optimize their own objectives. For instance, the work in \cite{sun2021two} explores the open-loop saddle point problem for a two-person zero-sum stochastic optimal control equation, while the study in \cite{li2023linear} examines the issue of Nash equilibria in N-player games. Additionally, the research in \cite{hambly2023policy} investigates the Nash equilibrium problem in N-player general-sum games using the policy gradient method. However, a significant limitation in much of the existing literature is its focus on deterministic coefficients, which do not capture the inherent uncertainties present in many real-world systems. In contrast, most real-world environments are affected by random disturbances, noise, and unpredictable changes \cite{tang2003general}. By incorporating stochastic coefficients into these models, we can more accurately reflect the complexities and uncertainties of actual systems, ultimately enhancing the model's applicability and relevance to real-world scenarios \cite{li2018indefinite}.

To fill a gap in current research, this paper focuses on exploring a class of two-person zero-sum stochastic control problems with random coefficients. This problem integrates the characteristics of Brownian motion, fractional Brownian motion, and Markov chains, and conducts an in-depth analysis within an infinite time horizon, making this study quite compelling. Compared to \cite{wei2021infinite}, this paper does not rely on stringent assumptions but instead constructs a broader and more practical theoretical framework. Within this framework, we only need to satisfy basic properties within the field of stochastic control, including Lipschitz continuity, monotonicity, and boundedness, to apply it to a wider range of practical situations. However, due to the system we consider simultaneously incorporating Brownian motion, fBm, and Markov chains, there currently lacks a directly applicable It$\rm\hat{o}$'s formula. Moreover, the non-martingale and non-Markovian properties of fBm, combined with the randomness and dependency structure of Markov chains, collectively increase the difficulty in proving the existence and uniqueness of solutions to FBSDEs, particularly within an infinite time horizon where this challenge is further amplified. Additionally, the two-person zero-sum game problem discussed in this paper features random coefficients and an infinite time horizon, which further complicates the task of finding two optimal control functions. Compared to \cite{sun2021two,sun2014linear}, our work is more complex, requiring the handling of more randomness and the challenges posed by the infinite time horizon. To our knowledge, this is the first study to address this type of problem, thus possessing groundbreaking and exploratory significance.
This paper tackles these challenges, with the following contributions and innovations:

\begin{itemize}
	\item[$\bullet$] This paper introduces a class of two-person zero-sum stochastic control problems, defined over an infinite time horizon, that combines Markov chains with fractional Brownian motion, subject to monotonicity conditions. These new elements  significantly expands the scope of traditional stochastic control models, capturing more complex system dynamics that were previously unexplored in the literature. By integrating these  components, our approach not only introduces a new theoretical framework but also provides a powerful methodological tool to address challenges in modeling dynamic strategic interactions under uncertainty. Furthermore, the paper derives an innovative It$\rm\hat{o}$'s formula tailored specifically for the proposed system, filling a critical gap in the existing literature where such a formulation has yet to be developed.
	
	\item[$\bullet$] A key innovation of this work is the use of the Banach fixed-point theorem and approximation methods within an infinite time framwork to establish the existence and uniqueness of solutions for FBSDEs, as well as for the coupled equations. We rigorously prove that, under appropriate conditions, the solutions to these highly complex stochastic equations are not only well-defined but also uniquely determined. This result provides a crucial step for further theoretical  analysis and the development of optimal control strategies.
	
	\item[$\bullet$] This paper introduces a optimal control representation for the two-person zero-sum problem by employing saddle point theory in a novel context. In stochastic systems, it is essential to capture the intricate interplay between the participants’ strategies and the inherent randomness, such as that induced by Markov chains and fractional Brownian motion. This innovative approach offers a comprehensive and systematic framework for characterizing the optimal strategies of both participants in such stochastic environments, marking a significant advancement in the analysis of dynamic strategic interactions under uncertainty.
	
	\item[$\bullet$] we delve into the impact of the cross term $ S(\cdot) $ on the system's behavior, revealing a critical innovation in understanding the coupling dynamics. When $ S(\cdot) $ is non-zero, it introduces a complex interdependence between the state variables and the control actions of both players, significantly complicating the optimal control strategies. In contrast, when $ S(\cdot) $ is zero, the system simplifies, leading to a more tractable characterization of the optimal strategies. This distinction provides profound insights into how varying degrees of coupling influence the design of optimal control solutions in stochastic environments, offering a powerful framework for tackling more complex, real-world scenarios.
\end{itemize}

The remainder of the paper is organized as follows. Section 2 presents the necessary background and notation. Section 3 establishes the existence and uniqueness of solutions for forward stochastic differential equations. Section 4 addresses the existence and uniqueness of solutions for backward stochastic differential equations. Section 5 proves the existence and uniqueness of solutions for the coupled equations. Section 6 provides a representation of the optimal control problem and a discussion on the coefficient of the cross term. Finally, Section 7 concludes the paper.

\section{Preliminaries and notations}

Let $\mathbb{R}^{n}$ be the $ n $-dimensional Euclidean space equipped with the Euclidean inner product $\langle \cdot , \cdot \rangle $. The induced norm is denoted by $| \cdot | $. Let $\mathbb{R}^{m\times n}$ be the set of all $m \times n$ matrices and $ \mathbb{S}^{n} $ be the set of all $ n \times n $ symmetrical matrices. $A^{T}$ denotes the transpose of $A$.

Let $(\Omega, \mathcal{F}, \mathbb{F}, \mathbb{P})$ be a complete filtered probability space on which is defined a $d$-dimensional Brownian motion $\left \{W(t), t \geq 0 \right \} $. Let $\left \{B^{H}(t), t \geq 0 \right \}$ be a fractional Brownian motion in $(\Omega, \mathcal{F}, \mathbb{F}, \mathbb{P})$, where the Hurst index  $H \in\left(\frac{1}{2}, 1\right) $, and each dimension is independent of each other. $\alpha(t)$  is a right continuous Markov chain on  $(\Omega, \mathcal{F}, \mathbb{F}, \mathbb{P})$, setting in  $\mathbb{S}=\{1,2, \cdots, m\}$. For each pair $(i_{0},j_{0})\in\mathbb{S} \times \mathbb{S}, i_{0}\ne j_{0}, 0\le t <+\infty $, we define $\left[M_{i_{0} j_{0}}\right](t)=\sum\limits_{0 \leq s \leq t} \chi\left(\alpha(s-)=i_{0}\right) \chi\left(\alpha(s)=j_{0}\right)$, $\left\langle M_{i_{0} j_{0}}\right\rangle(t)=\int_{0}^{t} q_{i_{0} j_{0}} \chi\left(\alpha(s-)=i_{0}\right) d s$ and $M_{i_{0} j_{0}} =\left[M_{i_{0} j_{0}}\right](t)-\left\langle M_{i_{0} j_{0}}\right\rangle(t)$, where $\chi\left(\alpha(s-)=i_{0}\right)$ means that $\alpha(s-)=1$ for $i_{0}$ and 0 for other cases. Further, we assume that $\alpha(t), W(t)$ and $B^{H}(t)$ are independent of each other.

Let $t \in [0,+\infty)$ and $K\in \mathbb{R}$, we introduce some spaces:
\begin{itemize}
	\item[$\bullet$]$L_{\mathcal{F}_{t}}^{2}\left(\Omega ; \mathbb{R}^{n}\right)$  is the set of  $\mathbb{R}^{n} $-valued  $\mathcal{F}_{t}$ -measurable random variables  $\xi$  such that
	$\|\xi\|_{L_{\mathcal{F}_{t}}^{2}\left(\Omega ; \mathbb{R}^{n}\right)}:=\left\{E\left[|\xi|^{2}\right]\right\}^{1 / 2}<+\infty $.
	\item[$\bullet$]$L_{\mathcal{F}_{t}}^{+\infty}\left(\Omega ; \mathbb{R}^{m \times n}\right)$  is the set of  $\mathbb{R}^{m \times n} $-valued  $\mathcal{F}_{t} $-measurable random variables  $ A $  such that
	$ \|A\|_{L_{\mathcal{F}_{t}}^{+\infty}\left(\Omega ; \mathbb{R}^{m \times n}\right)}:=\underset{\omega \in \Omega}{\operatorname{esssup}}\|A(\omega)\|<+\infty  $.
	\item[$\bullet$]$L_{\mathbb{F}}^{2, K}\left(0, +\infty ; \mathbb{R}^{n}\right)$  is the set of  $\mathbb{R}^{n} $-valued  $\mathbb{F} $-progressively measurable processes  $f(\cdot)$  such that
	$ \|f(\cdot)\|_{L_{\mathbb{F}}^{2, K}\left(0, +\infty ; \mathbb{R}^{n}\right)}:=\left\{E \int_{0}^{+\infty}\left|f(s) e^{K s}\right|^{2} d s\right\}^{1 / 2}<+\infty $.
	\item[$\bullet$]$L_{\mathbb{F}}^{+\infty}\left(0, +\infty ; \mathbb{R}^{m \times n}\right)$  is the set of  $\mathbb{R}^{m \times n} $-valued  $\mathbb{F} $-progressively measurable processes  $A(\cdot)$  such that
	$ \|A(\cdot)\|_{L_{\mathbb{F}}^{+\infty}\left(0, +\infty ; \mathbb{R}^{m \times n}\right)}:=\underset{(s, \omega) \in[0, +\infty) \times \Omega}{\operatorname{esssup}}\|A(s, \omega)\|<+\infty  $.
	\item[$\bullet$]$L^{2,H}(0,+\infty;\mathbb{R}^{n})$ is the set of  $\mathbb{R}^{n} $-valued  $f(\cdot)$ variable such that $\|f\|_{L^{2,H}(0,+\infty;\mathbb{R}^{n})}\\:=\int_{0}^{+\infty} \int_{0}^{+\infty} f(u) f(s) \varphi_{H}(u, s) d s d u<+\infty$, where $\varphi_{H}(u, s)=H(2 H-1)|u-s|^{2 H-2}$, $0\le s \le u<+\infty$.
\end{itemize}

In this paper, we study the optimal control problem for stochastic linear quadratic programming, integrating the dynamics of standard Brownian motion, fractional Brownian motion, and Markov chains over an infinite time.
\begin{eqnarray}
	\left\{\begin{aligned}
	&d x(s)=\big[A(s,\alpha(s)) x(s)+B_{1}(s,\alpha(s)) u_{1}(s)+B_{2}(s,\alpha(s)) u_{2}(s)\big] d s\\&\quad \quad \quad+H(s)dB^{H}(s)+\sum_{i=1}^{d}\big[C_{i}(s,\alpha(s)) x(s)+D_{1i}(s,\alpha(s)) u_{1}(s)\\&\quad \quad \quad+D_{2i}(s,\alpha(s)) u_{2}(s)\big] d W_{i}(s),\quad s \in[t_{0}, +\infty), \\
	&x(t_{0})=x_{t_{0}},\quad \alpha(t_{0})=i.
	\end{aligned}\right.\label{xx}
\end{eqnarray}
where $t_{0}\in [0,+\infty)$ is the initial time, $x_{t_{0}}$ is the initial state. $A(\cdot, \cdot), B_{1}(\cdot, \cdot),B_{2}(\cdot, \cdot), C_{i}(\cdot, \cdot)$ and $D_{1i}(\cdot, \cdot)(i=1,2,\cdots, d),D_{2i}(\cdot, \cdot)(i=1,2,\cdots, d)$ are bounded matrix-valued stochastic processes, $H(\cdot)$ is bounded matrix-valued function. $W=(W_{1}(\cdot),W_{2}(\cdot),\cdots,W_{d}(\cdot))^{T}$ is a $d$-dimensional Brownian motion, $B^{H}(\cdot)$ is a fractional Brownian motion, $x(\cdot)$ valued in $\mathbb{R}^{n}$ is the state process, $u(\cdot)$ valued in $\mathbb{R}^{m}$ is the control process.
The corresponding cost function is shown below
\begin{eqnarray}
	\quad \label{jj} \\
	 J^{K}\left(t_{0}, x_{t_{0}} ; u_{1}(\cdot),u_{2}(\cdot)\right)=\frac{1}{2} E \int_{t_{0}}^{+\infty} e^{2 K s}\left\langle \Pi(s,\alpha(s))\left(\begin{array}{l}
	x(s) \\
	u_{1}(s) \\ 
	u_{2}(s)
	\end{array}\right),\left(\begin{array}{l}
	x(s) \\
	u_{1}(s) \\
	u_{2}(s)
	\end{array}\right)\right\rangle d s.\nonumber
\end{eqnarray}
where $\Pi(s,\alpha(s))=\left(\begin{array}{lll}
Q(s,\alpha(s)) & S_{1}(s,\alpha(s))^{T} & S_{2}(s,\alpha(s))^{T} \\
S_{1}(s,\alpha(s)) & R_{11}(s,\alpha(s)) & R_{12}(s,\alpha(s)) \\
S_{2}(s,\alpha(s)) & R_{21}(s,\alpha(s)) & R_{22}(s,\alpha(s))
\end{array}\right)$, $K\in \mathbb{R}$ is a constant, $Q(\cdot, \cdot), S_{1}(\cdot, \cdot), S_{1}(\cdot, \cdot), R_{11}(\cdot, \cdot),R_{12}(\cdot, \cdot),R_{11}(\cdot, \cdot),R_{21}(\cdot, \cdot),R_{22}(\cdot, \cdot)$ are bounded matrix-valued stochastic processes.

 The function (\ref{jj}) can be interpreted as the loss for Player 1 and the corresponding gain for Player 2. In the context of this two-person zero-sum stochastic optimal control problem, Player 1 seeks to choose a control that minimizes their own loss, while Player 2 aims to maximize their gain. The optimal outcome for both players occurs when the control pair is such that neither player can improve their outcome by unilaterally changing their strategy, assuming the other player's control remains fixed. This equilibrium point, where both players' strategies are mutually optimal, is referred to as  saddle point, and it is mathematically defined by the following inequalities:

$$J^{K}\left( u_{1}^{*},u_{2}\right)\le J^{K}\left( u_{1}^{*},u_{2}^{*}\right)\le J^{K}\left( u_{1},u_{2}^{*}\right), \quad \forall \, u_{1},u_{2},$$
where $u_{1}^{*},u_{2}^{*}$ represent the optimal control.

Given that this paper focuses on stochastic equations driven by fractional Brownian motion over an infinite time horizon, it is crucial to first demonstrate the existence and uniqueness of the solution to equation (\ref{xx}) before proceeding with the derivation of the optimal control.
\section{ Infinite horizon SDE}
Consider the following infinite horizon SDE:
\begin{eqnarray}
    \left\{\begin{aligned}
    &d x(s)=b(s, x(s), \alpha(s) d s+\sum_{i=1}^{d} \sigma_{i}(s, x(s), \alpha(s)) d W_{i}(s)+\gamma(s) dB^{H}(s), \\
    &x(t_{0})=x_{t_{0}} \quad \alpha(t_{0})=i,
    \end{aligned}\right.\label{Xequ}
\end{eqnarray}
where $s \in[t_{0}, +\infty)$, $b$ and $\sigma_{i} (i=1,2,\cdots,d)$ are mappings from $[t_{0}, +\infty)\times \Omega \times \mathbb{R}^{n}$ to $\mathbb{R}^{n}$, $\gamma$ is mapping from $[t_{0}, +\infty) \times \mathbb{R}^{n}$ to $\mathbb{R}^{n}$. For convenience, we denote $\sigma=(\sigma^{T}_{1},\sigma^{T}_{2},\cdots,\sigma^{T}_{d})^{T}$.

\textbf{Assumption A}. (i)  $x_{t_{0}} \in L_{\mathcal{F}_{t}}^{2}\left(\Omega ; \mathbb{R}^{n}\right) $. For any  $x \in \mathbb{R}^{n} $, the processes  $b(\cdot, x,\cdot)$  and  $\sigma(\cdot, x,\cdot)$  are  $\mathbb{F}$-progressively measurable. Moreover, there exists a constant  $K \in \mathbb{R}$  such that  $b(\cdot, 0,\cdot) \in L_{\mathbb{F}}^{2, K}\left(0, +\infty ; \mathbb{R}^{n}\right)$  and  $\sigma(\cdot, 0,\cdot) \in L_{\mathbb{F}}^{2, K}\left(0, +\infty ; \mathbb{R}^{n d}\right) $. And $\gamma(\cdot) \in L^{2,H}(0,+\infty;\mathbb{R}^{n})$, $E\left[H(2H-1)\int_{0}^{T}\int_{0}^{s}|u-s|^{2H-2}\gamma(s)e^{2 K s}duds\right]<+\infty$.
\\(ii) The functions $ b $ and $ \sigma $ are uniformly Lipschitz continuous with respect to $ x $:
$|b(s, x,i_{0})-b(s, \bar{x},i_{0})| \leq l_{b x}|x-\bar{x}|, |\sigma(s, x,i_{0})-\sigma(s, \bar{x},i_{0})| \leq l_{\sigma x}|x-\bar{x}|$,
where $l_{b x}$ and $l_{\sigma x}$ are Lipschitz constants, $i_{0}\in \mathbb{S}$.
\\(iii) The function $ b $ is monotonic with respect to $ x $:
$\langle b(s, x,i_{0})-b(s, \bar{x},i_{0}), x-\bar{x}\rangle \leq-\kappa_{x}|x-\bar{x}|^{2}$,
where $\kappa_{x}\in \mathbb{R}$ is a constant, $i_{0}\in \mathbb{S}$.
\begin{lemma}
  Suppose that Assumption A hold. Let $X(t)$ be solution of (\ref{Xequ}). Then, the It$\hat{o}$ formula applied to $ X(t) $ gives:
\begin{equation}	
	\begin{aligned}
		&f(t_{2}, X(t_{2}), \alpha(t_{2}))=f(t_{1}, X(t_{1}), \alpha(t_{1}))+\int_{t_{1}}^{t_{2}} \mathcal{L} f(s, X(s), \alpha(s)) d s \\&
		\quad \quad \quad+\sum_{i=1}^{d} \int_{t_{1}}^{t_{2}}\left\langle\nabla_{x} f(s, X(s), \alpha(s)), \sigma_{ i}(s, X(s), \alpha(s))\right\rangle d W_{i}(s) \\&
		\quad \quad \quad+ \int_{t_{1}}^{t_{2}}\left\langle\nabla_{x} f(s, X(s), \alpha(s)), \gamma(s)\right\rangle d B^{H}(s) \\&
		\quad \quad \quad+\sum_{i_{0} \neq j_{0}} \int_{t_{1}}^{t_{2}}\bigg[f\left(s, X(s), j_{0}\right)-f\left(s, X(s), i_{0}\right)\bigg] d M_{i_{0} j_{0}}(s), 
	\end{aligned}\label{lemma3.1}
\end{equation}
where $0 \leq t_{1} \leq t_{2} < +\infty$, $\mathcal{L} f\left(t, x, i_{0}\right)=\frac{\partial}{\partial t} f\left(t, x, i_{0}\right)+\mathcal{L}_{i_{0}} f\left(t, x, i_{0}\right)+Q f(t, x, \cdot)\left(i_{0}\right)$, $\mathcal{L}$ is the generator of system (\ref{Xequ}), $\mathcal{L}_{i_{0}} f\left(t, x, i_{0}\right)=b^{T}\left(t, x, i_{0}\right) \nabla_{x} f\left(t, x, i_{0}\right)+\\\frac{1}{2} tr\left(\nabla_{x x}^{2} f\left(t, x, i_{0}\right) A\left(t, x, i_{0}\right)\right)$, $A\left(t, x, i_{0}\right)=\sigma(t, x, i_{0})\sigma^{T}(t, x, i_{0})+H(2H-1)\int_{0}^{t}|u-t|^{2H-2}\gamma(u)du$, $\sigma=(\sigma^{T}_{1},\sigma^{T}_{2},\cdots,\sigma^{T}_{d})^{T}$.
\end{lemma}

\begin{proof}
First, we consider the following stochastic differential equation $dx(s) =b(s,x(s))ds+\sum_{i=1}^{d}\sigma_{i}(s,x(s)) dW_{i}(s)+\gamma(s) dB^{H}(s)$, where $ X(t) $ is assumed to be a solution to this equation  According to Theorem 4 in \cite{duncan2000stochastic}, we can apply the It$\rm\hat{o}$ formula with fractional Brownian motion on the time interval $(t_{1},t_{2})$, yielding the following expression
\begin{equation}	
\begin{aligned}
f(t_{2}, &X(t_{2}))=f(t_{1}, X(t_{1}))+\int_{t_{1}}^{t_{2}}\nabla_{t} f(s, X(s))d s+ \int_{t_{1}}^{t_{2}}\left\langle\nabla_{x} f(s, X(s)), \gamma(s)\right\rangle d B^{H}(s)\\&+\frac{1}{2}\int_{t_{1}}^{t_{2}}\nabla_{xx} f(s, X(s))\left[\sum_{i=1}^{d}\sigma_{i}^{2}(s,x(s))+H(2H-1)\int_{0}^{s}|u-s|^{2H-2}\gamma(s)du\right]d s \\&+\sum_{i=1}^{d} \int_{t_{1}}^{t_{2}}\left\langle\nabla_{x} f(s, X(s)), \sigma_{ i}(s, X(s))\right\rangle d W_{i}(s)  .
\end{aligned}
\end{equation}

Next, we examine stochastic differential equations with Markov chains, i.e., the It$\rm\hat{o}$ formula of (\ref{Xequ}).	
Consider the time interval $(t_{1},t_{2})$, let that $\rho_{1}, \rho_{2}, \cdots, \rho_{v}$ denote a sequence of stopping times associated with the Markov chain $\alpha(s)$, satisfying: $t_{1} = \rho_{0} < \rho_{1} < \rho_{2} <...< \rho_{v} < \rho_{v+1} = t_{2}$, where the moment $t_{2}$ may or may not coincide with a stop time. Then we have the following result

$\begin{aligned}
	&f(t_{2}, X(t_{2}), \alpha(t_{2}))-f(t_{1}, X(t_{1}), \alpha(t_{1})) \\& \quad
	=\sum_{n=0}^{v}\bigg[f\left(\rho_{n+1}, X\left(\rho_{n+1}\right), \alpha\left(\rho_{n+1}\right)\right)-f\left(\rho_{n+1}, X\left(\rho_{n+1}\right), \alpha\left(\rho_{n}\right)\right)\bigg] \\&\quad
	+\sum_{n=0}^{v}\bigg[f\left(\rho_{n+1}, X\left(\rho_{n+1}\right), \alpha\left(\rho_{n}\right)\right)-f\left(\rho_{n}, X\left(\rho_{n}\right), \alpha\left(\rho_{n}\right)\right)\bigg]. 
\end{aligned}$

When $s\in (\rho_{n},\rho_{n+1})$, it can be further rewritten as follows:

$\begin{aligned}
	 &f\left(\rho_{n+1}, X\left(\rho_{n+1}\right), \alpha\left(\rho_{n}\right)\right)-f\left(\rho_{n}, X\left(\rho_{n}\right), \alpha\left(\rho_{n}\right)\right) \\&\quad
	=\int_{\rho_{n}}^{\rho_{n+1}}\left[\frac{\partial}{\partial s}+\mathcal{L}_{\alpha\left(\rho_{n}\right)}\right] f\left(s, X(s), \alpha\left(\rho_{n}\right)\right) d s \\& \quad
	+\sum_{i=1}^{d} \int_{\rho_{n}}^{\rho_{n+1}}\left\langle\nabla_{x} f\left(s, X(s), \alpha\left(\rho_{n}\right)\right), \sigma_{i}\left(s, X(s), \alpha\left(\rho_{n}\right)\right)\right\rangle d W_{i}(s) \\&\quad
	 + \int_{\rho_{n}}^{\rho_{n+1}}\left\langle\nabla_{x} f\left(s, X(s), \alpha\left(\rho_{n}\right)\right), \gamma(s)\right\rangle d B^{H}(s) ,
\end{aligned}$
\\ where $\mathcal{L}$ is the generator of system (\ref{Xequ}), $\mathcal{L}_{i_{0}} f\left(t, x, i_{0}\right)=b^{T}\left(t, x, i_{0}\right) \nabla_{x} f\left(t, x, i_{0}\right)+\frac{1}{2} tr\left(\nabla_{x x}^{2} f\left(t, x, i_{0}\right) A\left(t, x, i_{0}\right)\right)$ and $A\left(t, x, i_{0}\right)=\sigma(t, x, i_{0})\sigma^{T}(t, x, i_{0})+H(2H-1)\int_{0}^{t}|u-t|^{2H-2}\gamma(t)du$, $\sigma=(\sigma^{T}_{1},\sigma^{T}_{2},\cdots,\sigma^{T}_{d})^{T}$.

Since $\alpha(s)=\alpha(\rho_{n})$ for $s\in (\rho_{n},\rho_{n+1}), 0\le n\le v$ and noting that $\alpha(\rho_{n})=\alpha(\rho_{n+1}-)$, then

$\begin{aligned}
&f\left(t_{2}, X\left(t_{2}\right), \alpha\left(t_{2}\right)\right)-f\left(t_{1}, X\left(t_{1}\right), \alpha\left(t_{1}\right)\right) \\&\quad
=\int_{t_{1}}^{t_{2}}\left[\frac{\partial}{\partial s}+\mathcal{L}_{\alpha\left(s\right)}\right] f\left(s, X(s), \alpha\left(s\right)\right) d s \\& \quad
+\sum_{i=1}^{d} \int_{t_{1}}^{t_{2}}\left\langle\nabla_{x} f\left(s, X(s), \alpha\left(s\right)\right), \sigma_{i}\left(s, X(s), \alpha\left(s\right)\right)\right\rangle d W_{i}(s) \\&\quad
+ \int_{t_{1}}^{t_{2}}\left\langle\nabla_{x} f\left(s, X(s), \alpha\left(s\right)\right), \gamma(s)\right\rangle d B^{H}(s)\\&\quad
+\sum_{n=1}^{v+1}\bigg[f\left(\rho_{n}, X\left(\rho_{n}\right),\alpha\left(\rho_{n}\right)\right)-f\left(\rho_{n}, X\left(\rho_{n}\right),\alpha\left(\rho_{n}-\right)\right)\bigg] .
\end{aligned}$

As stated in Appendix A of \cite{nguyen2017milstein}, the final term in the above equation can be expressed as follows:
\\$\begin{aligned}
 &\sum_{n=1}^{v+1}\bigg[f\left(\rho_{n}, X\left(\rho_{n}\right), \alpha\left(\rho_{n}\right)\right)-f\left(\rho_{n}, X\left(\rho_{n}\right), \alpha\left(\rho_{n}-\right)\right)\bigg] \\&
=  \sum_{j_{0} \in M} \int_{t_{1}}^{t_{2}} q_{\alpha(s-), j_{0}}\bigg[ f\left(s, X(s), j_{0}\right)-f(s, X(s), \alpha(s-))\bigg] d s \\&\quad
 +\sum_{i_{0} \neq j_{0}} \int_{t_{1}}^{t_{2}}\bigg[ f\left(s, X(s), j_{0}\right)-f\left(s, X(s), i_{0}\right)\bigg] d M_{i_{0} j_{0}}(s) \\&
=  \int_{t_{1}}^{t_{2}} Q f(s, X(s), \cdot)(\alpha(s)) d s+\sum_{i_{0} \neq j_{0}} \int_{t_{1}}^{t_{2}}\bigg[f\left(s, X(s), j_{0}\right)-f\left(s, X(s), i_{0}\right)\bigg] d M_{i_{0} j_{0}}(s),
\end{aligned}$
\\ where $Q f(t, x, \cdot)\left(i_{0}\right)=\sum_{j_{0} \in M} q_{i_{0} j_{0}}[f\left(t, x, j_{0}\right)-f\left(t, x, i_{0}\right)]$.

As such, we derive that

$\begin{aligned}
&f(t_{2}, X(t_{2}), \alpha(t_{2}))=f(t_{1}, x(t_{1}), \alpha(t_{1}))+\int_{t_{1}}^{t_{2}} \mathcal{L} f(s, X(s), \alpha(s)) d s \\&
\quad+\sum_{i=1}^{d} \int_{t_{1}}^{t_{2}}\left\langle\nabla_{x} f(s, X(s), \alpha(s)), \sigma_{ i}(s, X(s), \alpha(s))\right\rangle d W_{i}(s) \\&
\quad+ \int_{t_{1}}^{t_{2}}\left\langle\nabla_{x} f(s, X(s), \alpha(s)), \gamma(s)\right\rangle d B^{H}(s) \\&
\quad+\sum_{i_{0} \neq j_{0}} \int_{t_{1}}^{t_{2}}\bigg[f\left(s, X(s), j_{0}\right)-f\left(s, X(s), i_{0}\right)\bigg] d M_{i_{0} j_{0}}(s), \quad 0 \leq t_{1} \leq t_{2} < +\infty ,
\end{aligned}$
\\where $\mathcal{L} f\left(t, x, i_{0}\right)=\frac{\partial}{\partial t} f\left(t, x, i_{0}\right)+\mathcal{L}_{i_{0}} f\left(t, x, i_{0}\right)+Q f(t, x, \cdot)\left(i_{0}\right)$.
\end{proof}

\textit{Remark}: To prove the existence and uniqueness of the solution to (\ref{Xequ}), It$\rm\hat{o}$'s formula is required. However, since the problem addressed in this paper is new, no existing form of It$\rm\hat{o}$'s formula is directly applicable. As a result, in Lemma 3.1, we develop a novel version of It$\rm\hat{o}$'s formula to aid in the subsequent proof.

\begin{theorem}
	Suppose that Assumption A holds, then there exists a unique solution to SDE (\ref{Xequ}) on $[0,+\infty)$.
\end{theorem}
\begin{proof}
	
For any $T\in [0,+\infty)$, consider the process space $L^{2}(T)$, which consists of processes $\left \{ \mathcal{F}_{t} ,t\ge 0 \right \} $ adapted to the process and satisfies $\left\| X \right\| _{L^{2}(T) } =\sqrt{E\int_{0}^{T}X^{2} (t)dt } <+\infty $. And the space consisting of the full process $X=\left \{ X_{t} , 0\le t\le T \right \}$. It follows that $(L^{2}(T),\left\| \cdot \right\|_{L^{2}(T) })$ is a Banach space and $ \left\|X \right\|=\sqrt{\sup_{t\in [0,T]}e^{-at}  E\int_{0}^{t}X^{2} (s)ds } $ is an equivalent norm for any given $a>0$. 

Given any $x=\left \{ x(t), 0 \le t \le T, x_{0}=\zeta \right \} \in L^{2} (T)$, define the process
\begin{equation}
X(t)=\zeta +\int_{0}^{t} b(s,x(s), \alpha(s))ds+\sum_{i=1}^{d}\int_{0}^{t}  \sigma_{i} (s,x(s),\alpha(s))dW_{i}(s)+\int_{0}^{t}\gamma(s)dB^{H}(s).
\label{Xjie}
\end{equation}
Under Assumption A, it is established that
 $$E\int_{0}^{T} |b(s,x(s),\alpha(s))|^{2}ds\le CE\int_{0}^{T}\left[1+|x(s)|^{2}ds\right]<+\infty, $$
$$E\int_{0}^{T} |\sigma (s,x(s),\alpha(s))|^{2} ds\le CE\int_{0}^{T}\left[1+|x(s)|^{2} ds\right]<+\infty $$
and building on Lemma 4.1 in \cite{fan2023asymptotic}, it follows that
$$E\left[\int_{0}^{t}\gamma(s)dB^{H}(s)\right]^{2}\le CE\int_{0}^{T}|\gamma(s)|^{2}ds< +\infty,$$
where $C$ is a constant. As a result, for any $\forall \, t \in [0,T]$, $X(t)$ is well-defined, and there is

$\begin{aligned}
	&E\left[|X(t)|^{2} \right]\le 4E\zeta ^{2} +4TE\int_{0}^{t} |b (s,x(s),\alpha(s)|^{2} ds+4E\int_{0}^{t} |\sigma (s,x(s), \alpha(s)|^{2} ds\\&\quad \quad \quad \quad \quad+4E\left[\int_{0}^{t}\gamma(s)dB^{H}(s)\right]^{2}
	\le 4E\zeta ^{2}+C_{1}E\int_{0}^{T}\left[1+|x(s)|^{2}\right] ds+C_{t},
	\end{aligned}$
\\where $\sigma=(\sigma^{T}_{1},\sigma^{T}_{2},\cdots,\sigma^{T}_{d})^{T}$. Thus $E\int_{0}^{T}|X(t)|^{2}<+\infty$. 

Define the mapping $A: x\mapsto  X$ over $L^{2} (T)$.
For any  $x^{(1)},x^{(2)}\in L^{2} (T)$, set $X^{(1)}=A(x^{(1)}),X^{(2)}=A(x^{(2)})$, we conclude that
\\$\begin{aligned}
\left\|A(x^{(1)})-A(x^{(2)})\right\|^{2}=\left\| X^{(1)}-X^{(2)}\right\|^{2}=\sup _{t \in[0, T]} e^{-a t} E \int_{0}^{t}\left|X^{(1)}(s)-X^{(2)}(s)\right|^{2} d s.
\end{aligned}$
\\Because	

$\begin{aligned}
	&X^{(1)}(t)-X^{(2)}(t)=\int_{0}^{t}\left[b(s, x^{(1)}(s),\alpha(s)-b(s, x^{(2)}(s),\alpha(s))\right] d s\\&\quad \quad \quad \quad \quad \quad \quad \quad+\int_{0}^{t}\left[\sigma(s, x^{(1)}(s),\alpha(s))-\sigma(s, x^{(2)}(s),\alpha(s))\right] d W(s),
\end{aligned}$
\\then taking the expectation of both sides gives

$\begin{aligned}
	&E\left[X^{(1)}(t)-X^{(2)}(t)\right]^{2}\\& \le 2E\left\{ \int_{0}^{t}\left[b(s, x^{(1)}(s),\alpha(s))-b(s, x^{(2)}(s),\alpha(s))\right] d s\right\}^{2}\\&+2E\left\{\int_{0}^{t}\left[\sigma(s, x^{(1)}(s),\alpha(s))-\sigma(s, x^{(2)}(s),\alpha(s))\right] d W(s)\right\}^{2}\\& \le 2 L^{2} T E\left[\int_{0}^{t}\left|x^{(1)}(s)-x^{(2)}(s)\right|^{2} d s\right]+2 L^{2} E\left[\int_{0}^{t}\left|x^{(1)}(s)-x^{(2)}(s)\right|^{2} d s\right]\\& = 2 L^{2}(T+1) E\left[\int_{0}^{t}\left|x^{(1)}(s)-x^{(2)}(s)\right|^{2} d s\right]\\& \le 2 L^{2}(T+1) e^{a t}\left\|x^{(1)}-x^{(2)}\right\|^{2},
\end{aligned}$,	
\\As a result, we derive	

$\begin{aligned}
	&\left\| A(x^{(1)})-A(x^{(2)})\right\|^{2}  =\left\|X^{(1)}-X^{(2)}\right\|^{2}\\&=\sup _{t \in[0, T]} e^{-a t} E \int_{0}^{t}\left|X^{(1)}(s)-X^{(2)}(s)\right|^{2} ds \\&
\le \sup _{t \in[0, T]} e^{-a t} 2 L^{2}(T+1)\left\|x^{(1)}-x^{(2)}\right\|^{2} \int_{0}^{t} e^{a s} ds \\&
	 \le \sup _{t \in[0, T]}\left[1-e^{-a t}\right] \frac{2 L^{2}(T+1)}{a}\left\| x^{(1)}-x^{(2)}\right\|^{2} \\&
	 \le \left[1-e^{-a T}\right] \frac{2 L^{2}(T+1)}{a}\left\| x^{(1)}-x^{(2)}\right\|^{2}.
	\end{aligned}$
	
	Take $a>0$ sufficiently large such that $\left[1-e^{-a T}\right] \frac{2 L^{2}(T+1)}{a}<1$, ensuring $A$ is a compression map on $L^{2}(T)$. By the Banach fixed-point theorem, we deduce that there exists a
	unique fixed point on $[0,T]$, which solves equation (\ref{Xjie}). Then, applying the extension theorem, equation (\ref{Xequ}) admits a unique solution according the extension theorem on $[0,+\infty)$.
	\end{proof}
	
	\textit{Remark}: The existence and uniqueness of solutions to classical stochastic control equations have already been established for finite time, meaning that most finite-time stochastic LQ problems do not require further consideration of solution existence. However, this paper focuses on a more complex class of novel stochastic equation within an infinite time horizon, involving Markov chains and fractional Brownian motion. Consequently, in Theorem 3.2, we use the Banach fixed-point theorem to first establish the existence and uniqueness of the solution to (\ref{Xequ}).
	
\begin{lemma}
	Let Assumption A hold. Assuming further that the solution $x$ of (\ref{Xequ}) belongs to $L_{\mathbb{F}}^{2, K}\left(0, +\infty ; \mathbb{R}^{n}\right)$, there exists
	\begin{equation}
		\lim_{u \to +\infty}E\left[|x(u)|^{2} e^{2K u}\right]=0.
	\end{equation}
\end{lemma}
\begin{proof}
	For any $u\in [0, +\infty)$, we apply It$\rm\hat{o}$'s formula of Lemma 3.1 to $|x(s)|^{2} e^{2K s}$ on the interval $[0,u]$:
\\$\begin{aligned}
E\left[|x(u)|^{2} e^{2K u}\right]&=  \left|x_{0}\right|^{2}+E \int_{0}^{u} 2\langle x(s), b(s, x(s),\alpha(s)\rangle e^{2K s} d s  
 +E \int_{0}^{u}\bigg[2K|x(s)|^{2}e^{2K s}\\&+|\sigma(s, x(s),\alpha(s))|^{2}e^{2K s}+H(2H-1)\int_{0}^{s}|r-s|^{2H-2}\gamma(s)e^{2K s}dr\bigg]  d s .
\end{aligned}$

By the Lipschitz continuity condition, we have

$\begin{aligned}
&E\left[|x(u)|^{2} e^{2K u}\right] \leq  \left|x_{0}\right|^{2}+\left(1+2|K|+2 C_{1}+2C_{2}^{2}\right) E \int_{0}^{u}|x(s)|^{2} e^{2K s} d s \\&
\quad \quad \quad \quad \quad \quad \quad  +E \int_{0}^{u}\bigg[|b(s, 0,\alpha(s))|^{2}+2|\sigma(s, 0,\alpha(s))|^{2}\bigg] e^{2K s} d s+C_{3},
\end{aligned}$
\\where $C_{1}, C_{2},C_{3}$ are  constants. Due to the fact that $  x(\cdot), b(\cdot, 0,\cdot) \in L_{\mathbb{F}}^{2, K}\left(0, +\infty ; \mathbb{R}^{n}\right)$ and $\sigma(\cdot, 0,\cdot) \in L_{\mathbb{F}}^{2, K}\left(0, +\infty ; \mathbb{R}^{n \times d}\right)$, the above inequality suggests the deterministic process $\left \{E\left[|x(u)|^{2} e^{2K u}\right]; u \ge 0 \right \} $ is bounded. Moreover, applying It$\rm\hat{o}$'s formula to $|x(s)|^{2} e^{2K s}$ on the interval $[u_{1},u_{2}]$:
\\$\begin{aligned}
&\left|E\left[\left|x\left(u_{2}\right)\right|^{2} e^{2K u_{2}}\right]-E\left[\left|x\left(u_{1}\right)\right|^{2} e^{2K u_{1}}\right]\right|\\& \leq  
 E \int_{u_{1}}^{u_{2}}\bigg[|b(s, 0,\alpha(s))|^{2}+2|\sigma(s, 0,\alpha(s))|^{2}+\left(1+2|K|+2 C+2 C^{2}\right)|x(s)|^{2}\bigg] e^{2K s} d s,
\end{aligned}$
\\which indicates that the process $\left \{E\left[|x(u)|^{2} e^{2K u}\right]; u \ge 0 \right \} $ is continuous. The inequality also shows that, $E\left[\left|x\left(u_{2}\right)\right|^{2} e^{2K u_{2}}\right]-E\left[\left|x\left(u_{1}\right)\right|^{2} e^{2K u_{1}}\right] \to 0$ as $u_{1}, u_{2} \to +\infty$, then $\lim\limits_{u \to +\infty}E\left[|x(u)|^{2} e^{2K u}\right]$ exists. What's more, due to $\int_{0}^{+\infty}E\left[|x(u)|^{2} e^{2K u}\right]du <+\infty$, we get $\lim\limits_{u \to +\infty}E\left[|x(u)|^{2} e^{2K u}\right]=0$.
\end{proof}

\begin{lemma}
	Assume that the coefficients $(x_{t_{0}},b,\sigma)$ satisfy Assumption A and $K<\kappa_{x}-(l_{\sigma x}^{2}/2)$. Then, the solution $x$ of (\ref{Xequ}) belongs to $L_{\mathbb{F}}^{2, K}\left(0, +\infty ; \mathbb{R}^{n}\right)$. Moreover,
	for any $\mu >0$, the following estimate holds:
	\begin{equation}
		\begin{aligned}
		&\left(2 \kappa_{x}-2 K-l_{\sigma x}^{2}-2 \mu\right) E \int_{t_{0}}^{+\infty}\left|x(s) e^{K s}\right|^{2} d s 
		 \leq E\bigg\{\left|x_{t_{0}} e^{K t_{0}}\right|^{2}\\&\quad \quad+\int_{t_{0}}^{+\infty}\left[\frac{1}{\mu}\left|b(s, 0,\alpha(s)) e^{K s}\right|^{2}+\left(1+\frac{l_{\sigma x}^{2}}{\mu}\right)\left|\sigma(s, 0,\alpha(s)) e^{K s}\right|^{2}\right] d s\bigg\}+C .
		\end{aligned}\label{xe1}
	\end{equation}
	Let $ \bar{x}_{t_{0}} $ be the solution of another set of coefficients $ (\bar{x}_{t_{0}},\bar{b}, \bar{\sigma}) $ of the SDE (\ref{Xequ}) satisfying Assumption A. Then, for any $\mu >0$, the following estimate is valid:
	\begin{equation}
		\begin{aligned}
		&\left(2 \kappa_{x}\right.  \left.-2 K-l_{\sigma x}^{2}-2 \mu\right) E \int_{t_{0}}^{+\infty}\left|(x(s)-\bar{x}(s)) e^{K s}\right|^{2} d s \\& \quad
		\leq  E\left\{\left|\left(x_{t_{0}}-\bar{x}_{t_{0}}\right) e^{K t_{0}}\right|^{2}+\int_{t_{0}}^{+\infty}\left[\frac{1}{\mu}\left|(b(s, \bar{x}(s),\alpha(s))-\bar{b}(s, \bar{x}(s),\alpha(s)) e^{K s}\right|^{2}\right.\right. \\&\quad
		 \left.\left.+\left(1+\frac{l_{\sigma x}^{2}}{\mu}\right)\left|(\sigma(s, \bar{x}(s),\alpha(s))-\bar{\sigma}(s, \bar{x}(s)),\alpha(s)) e^{K s}\right|^{2}\right] d s\right\} .
		\end{aligned}\label{xe2}
	\end{equation}
\end{lemma}

\begin{proof}
	For any $T>t_{0}$, applying It$\rm\hat{o}$'s formula to $|x(s) e^{K s}|^{2}$ on the interval $[t_{0},T]$:
	\\$\begin{aligned}
	&E  \left\{\left|x(T) e^{K T}\right|^{2}-2 K \int_{t_{0}}^{T}\left|x(s) e^{K s}\right|^{2} d s\right\} \\&
	=  E\bigg\{\left|x_{t_{0}} e^{K t_{0}}\right|^{2}+\int_{t_{0}}^{T}\bigg[2\langle x(s), b(s, x(s),\alpha(s))\rangle e^{2 K s}+|\sigma(s, x(s),\alpha(s))|^{2}e^{2 K s}\\&\quad+H(2H-1)\int_{0}^{s}|u-s|^{2H-2}\gamma(s)e^{2 K s}du\bigg]  d s\bigg\} \\&
	\leq  E\bigg\{\left|x_{t_{0}} e^{K t_{0}}\right|^{2}+\int_{t_{0}}^{T}\bigg[2\langle x(s), b(s, x(s),\alpha(s))-b(s, 0,\alpha(s))\rangle+2|x(s)||b(s, 0,\alpha(s))| \\&\quad
	+(|\sigma(s, x(s),\alpha(s))-\sigma(s, 0,\alpha(s))|+|\sigma(s, 0,\alpha(s))|)^{2}\bigg] e^{2 K s} d s\bigg\}+C\\&
	\leq E\bigg\{\left|x_{t_{0}} e^{K t_{0}}\right|^{2}+\int_{t_{0}}^{T}\bigg[-2 \kappa_{x}|x(s)|^{2}+2|x(s)||b(s, 0,\alpha(s))|\\& \quad +\left(l_{\sigma x}|x(s)|+|\sigma(s, 0,\alpha(s))|\right)^{2}\bigg] e^{2 K s} d s\bigg\} +C\\&
	\leq E\bigg\{\left|x_{t_{0}} e^{K t_{0}}\right|^{2}+\int_{t_{0}}^{T}\bigg[\left(l_{\sigma x}^{2}+2 \mu-2 \kappa_{x}\right)|x(s)|^{2}+\frac{1}{\mu}|b(s, 0,\alpha(s))|^{2}\\&\quad +\left(1+\frac{l_{\sigma x}^{2}}{\mu}\right)|\sigma(s, 0,\alpha(s))|^{2}\bigg] e^{2 K s} d s\bigg\}+C. 
	\end{aligned}$
	
The final inequality follows from the application of the inequality $ 2ab \leq \mu | a|^{2} + (1/\mu )| b|^{2}$. Therefore, by Lemma 3.3, as $T\to +\infty$, we obtain
\\$\begin{aligned}
&\left(2 \kappa_{x}-2 K-l_{\sigma x}^{2}-2 \mu\right) E \int_{t_{0}}^{+\infty}\left|x(s) e^{K s}\right|^{2} d s 
 \leq E\bigg\{\left|x_{t_{0}} e^{K t_{0}}\right|^{2}\\&\quad \quad+\int_{t_{0}}^{+\infty}\left[\frac{1}{\mu}\left|b(s, 0,\alpha(s)) e^{K s}\right|^{2}+\left(1+\frac{l_{\sigma x}^{2}}{\mu}\right)\left|\sigma(s, 0,\alpha(s)) e^{K s}\right|^{2}\right] d s\bigg\}+C .
\end{aligned}$

By virtue of the density of real numbers, the condition  $K<\kappa_{x}-\left(l_{\sigma x}^{2} / 2\right)$  necessitates the existence of $\mu \in\left(0, \kappa_{x}-K-\left(l_{\sigma x}^{2} / 2\right)\right)$. Therefore, the estimate (\ref{xe1}) along with Assumption A ensures that  $x(\cdot) \in L_{\mathbb{F}}^{2, K}\left(0, +\infty ; \mathbb{R}^{n}\right)$.

We denote
$\widehat{x}_{t_{0}}=x_{t_{0}}-\bar{x}_{t_{0}},  \widehat{b}(s, x,\alpha(s))=b(s, x+\bar{x}(s),\alpha(s))-\bar{b}(s, \bar{x}(s),\alpha(s)),  \widehat{\sigma}(s, x,\\\alpha(s))=\sigma(s, x+\bar{x}(s),\alpha(s))-\bar{\sigma}(s, \bar{x}(s),\alpha(s))$
for any  $(s, \omega, x) \in[t_{0}, +\infty) \times \Omega \times \mathbb{R}^{n} $. Then, it is straightforward to confirm that the coefficients  $(\widehat{x}, \widehat{b}, \widehat{\sigma}) $ also satisfy Assumption A with the same constants  $K, l_{b x}, l_{\sigma x}$, and  $\kappa_{x} $. Moreover, we proceed to verify that  $\widehat{x}(\cdot)=x(\cdot)-\bar{x}(\cdot)$  satisfies the following SDE with the corresponding coefficients  $\left(\widehat{x}_{t_{0}}, \widehat{b}, \widehat{\sigma}\right) $:

$\left\{\begin{aligned}
&d \widehat{x}(s)=\widehat{b}(s, \widehat{x}(s),\alpha(s)) d s+\sum_{i=1}^{d} \widehat{\sigma}_{i}(s, \widehat{x}(s),\alpha(s)) d W_{i}(s), \quad s \in[t_{0}, +\infty), \\&
\widehat{x}(t_{0})=\widehat{x}_{t_{0}},\quad \alpha(t_{0})=i.
\end{aligned}\right.$
\\By employing a similar argument as in (\ref{xe1}), we can derive (\ref{xe2}).
\end{proof}

 The existence and uniqueness of the solution to the forward SDE (\ref{Xequ}) are established through Theorem 3.2. The next step is to derive the corresponding adjoint equation from (\ref{Xequ}) and establish the existence and uniqueness of its solution. 
\section{Infinite horizon BSDE}
Consider the following infinite horizon backward SDE (BSDE):
\begin{equation}	
\begin{aligned}
	dy(s)=g(s, y(s), z(s),r(s),f(s),\alpha(s ))d s+\sum_{i=1}^{d} z_{i}(s)d W_{i}(s)+r(s)dB^{H}(s)+f(s)\cdot d\mathcal{M}(s),\label{yequ}
\end{aligned}
\end{equation}
where $s \in[t_{0}, +\infty)$, $g$ is a mapping from $[t_{0},+\infty)\times \Omega \times \mathbb{R}^{n} \times \mathbb{R}^{nd} \times \mathbb{R}^{n} \times \mathbb{R}^{n}$ to $\mathbb{R}^{n}$, and we denote $z=(z^{T}_{1},z^{T}_{2},\cdots,z^{T}_{d})^{T}$.

\textbf{Assumption B}. (i) For any $(y, z,r,f) \in \mathbb{R}^{n} \times \mathbb{R}^{nd} \times \mathbb{R}^{n} \times \mathbb{R}^{n}$, the process $g(\cdot , y, z,r,f,\cdot)$ is $\mathbb{F}$ -progressively measurable, and satisfies $g(s,0,0,0,0,i_{0})\in L_{\mathbb{F}}^{2, K}\\\left(0, +\infty ; \mathbb{R}^{n+nd+n+n}\right)$, $E\left[H(2H-1)\int_{0}^{T}\int_{0}^{s}|u-s|^{2H-2}r(s)e^{2 K s}duds\right]<+\infty$, where $i_{0}\in \mathbb{S}$.
\\(ii) $ g $ is uniformly Lipschitz continuous with respect to $ (y,z,r,f) $:
$|g(s, y,z,r,f,i_{0})-g(s, \bar{y},\bar{z},\bar{r},\bar{f},i_{0})| \leq l_{gy}|y-\bar{y}|+l_{gz}|z-\bar{z}|+l_{gr}|r-\bar{r}|+l_{gf}|f-\bar{f}|$,
where $l_{gy},l_{gz},l_{gr},l_{gf}$ are Lipschitz constants, $i_{0}\in \mathbb{S}$.

\begin{lemma}
Let Assumption B hold and $K\in\mathbb{R}$. We further suppose  that the solution $(y,z,r,f)$ of (\ref{yequ}) belongs to $L_{\mathbb{F}}^{2, K}\left(0, +\infty ; \mathbb{R}^{n}\right) \times L_{\mathbb{F}}^{2, K}\left(0, +\infty ; \mathbb{R}^{nd}\right) \times L_{\mathbb{F}}^{2, K}(0, +\infty ;$
 \\$\mathbb{R}^{n}) \cap L^{2,H}(0,+\infty;\mathbb{R}^{n}) \times L_{\mathbb{F}}^{2, K}\left(0, +\infty ; \mathbb{R}^{n}\right)$. Thus, we arrive at
\begin{equation}
\lim_{u \to +\infty}E\left[|y(u)|^{2} e^{2K u}\right]=0.
\end{equation}	
\end{lemma}

\begin{proof}
	Considering that $ (y,z,r,f) $ is a solution to BSDE (\ref{yequ}), we find for any $u\in [0,+\infty)$,
	\\	$\begin{aligned}
	&y(u)=  \int_{u}^{+\infty}\left[g(s, y(s), z(s), r(s),f(s),\alpha(s))\right] d s+\int_{u}^{+\infty} z(s) d W(s) +\int_{u}^{+\infty}  r(s)dB^{H}(s) \\&\quad \quad \quad+\int_{u}^{+\infty}f(s)\cdot d\mathcal{M}(s)\\& \quad \quad
	=  \int_{0}^{+\infty}\left[g(s, y(s), z(s), r(s),f(s),\alpha(s))\right] d s+\int_{0}^{+\infty} z(s) d W(s)+\int_{0}^{+\infty} r(s) dB^{H}(s)\\&\quad \quad \quad+\int_{0}^{+\infty}f(s)\cdot d\mathcal{M}(s) 
	-\int_{0}^{u}\left[g(s, y(s), z(s), r(s),f(s),\alpha(s))\right] d s-\int_{0}^{u} z(s) d W(s)\\&\quad \quad \quad-\int_{0}^{u} r(s) dB^{H}(s)-\int_{0}^{u}f(s)\cdot d\mathcal{M}(s) \\&\quad \quad
	=  y(0)-\int_{0}^{u}\left[g(s, y(s), z(s), r(s),f(s),\alpha(s))\right] d s-\int_{0}^{u} z(s) d W(s)-\int_{0}^{u} r(s) dB^{H}(s)\\&\quad \quad \quad-\int_{0}^{u}f(s)\cdot d\mathcal{M}(s).
	\end{aligned}$
	
	Consequently, the process $ y $ is viewed as an adapted solution to a forward SDE. Referring to (4.29) in \cite{nguyen2017milstein}, one deduces that the square of the last term 
	$\int_{0}^{u}f(s)\cdot d\mathcal{M}(s)$ is bounded. The result follows by applying the proof method of Lemma 3.3.
\end{proof}

\begin{lemma}
	Let $ g $ fulfill the requirements of  Assumption B. Moreover, it is assumed that\\ $\left \langle y,g(s,y(s),z(s),r(s),f(s),i_{0})-g(s,0,0,0,0,i_{0})  \right \rangle \ge L(-|y(s)|^{2}-|z(s)|^{2}+|r(s)|^{2}+|f(s)|^{2}) $, where $0<L<\min\{1/2,K-\mu/2\}$. Let $(y,z,r,f)$ be a solution to (\ref{yequ}). Then, for any $\mu>0$, it can be shown that
		\begin{equation}
		\begin{aligned}
		&E\bigg\{|y(t_{0})e^{Kt_{0}}|^{2}+\int_{t_{0}}^{+\infty}\bigg[(2K-2L-\mu)|y(s)e^{Ks}|^{2}+(1-2L)|z(s)e^{Ks}|^{2}\\ &\quad 
		+2L|f(s)e^{Ks}|^{2}+2L|r(s)e^{Ks}|^{2}\bigg] d s\bigg\} 
		\le E\left \{\int_{t_{0}}^{+\infty} \frac{1}{\mu}|g(s,0,0,0,0,\alpha(s))e^{Ks}|^{2} \right \}+C.\label{ye1}
		\end{aligned}
	\end{equation}
Let $ (\bar{y},\bar{z},\bar{r},\bar{f}) $ be the solution corresponding to another choice for coefficient $  \bar{g}$ of the BSDE (\ref{yequ}) that adheres to  Assumption B. Hence for any $\mu >0$, we conclude	
	\begin{equation}
		\begin{aligned}
		&E\bigg\{|(y(t_{0})-\bar{y}(t_{0}))e^{Kt_{0}}|^{2}+\int_{t_{0}}^{+\infty}\bigg[(2K-2L-\mu)|(y(s)-\bar{y}(s))e^{Ks}|^{2}\\&\quad
		+(1-2L)|(z(s)-\bar{z}(s))e^{Ks}|^{2}+2L|(f(s)-\bar{f}(s))e^{Ks}|^{2}+2L|(r(s)-\bar{r}(s))e^{Ks}|^{2}\bigg] d s\bigg\} \\&
		\le E\left \{\int_{t_{0}}^{+\infty} \frac{1}{\mu}|(g(s,\bar{y}(s),\bar{z}(s),\bar{r}(s),\bar{f}(s),\alpha(s))-\bar{g}(s,\bar{y}(s),\bar{z}(s),\bar{r}(s),\bar{f}(s),\alpha(s))e^{Ks}|^{2} \right \}.\label{ye2}
		\end{aligned}
	\end{equation}
\end{lemma}
\begin{proof}
	For any $T>t_{0}$, applying It$\rm\hat{o}$'s formula to $|y(s) e^{K s}|^{2}$ on the interval $[t_{0},T]$:
	\\	$\begin{aligned}
		&E\left \{|y(t_{0})e^{Kt_{0}}|^{2}  \right \}\\&
		=E\bigg \{|y(T)e^{KT}|^{2}-\int_{t_{0}}^{T} \bigg[2\left \langle  y(s),g(s,y(s),z(s),r(s),f(s),\alpha(s))\right \rangle e^{2Ks}
		\\&\quad+2K|y(s)e^{Ks}|^{2} +|z(s)e^{Ks}|^{2}+H(2H-1)\int_{0}^{s}|u-s|^{2H-2}r(s)e^{2Ks}du\bigg]ds\bigg\}\\&
		\le E\bigg \{ |y(T)e^{KT}|^{2} -\int_{t_{0}}^{T}\bigg[2K|y(s)e^{Ks}|^{2}+|z(s)e^{Ks}|^{2}+2\left \langle y(s),g(s,0,0,0,0,\alpha(s)) \right \rangle e^{2Ks}\\&\quad
		+2\left \langle y,g(s,y(s),z(s),r(s),f(s),\alpha(s))-g(s,0,0,0,0,\alpha(s)) \right \rangle e^{2Ks}\bigg]ds\bigg \}+C\\&
		\le E\bigg \{|y(T)e^{KT}|^{2}-\int_{t_{0}}^{T}\bigg[ 2L(-|y(s)|^{2}-|z(s)|^{2}+|r(s)|^{2}+|f(s)|^{2})e^{2Ks}+2K|y(s)e^{Ks}|^{2} \\&\quad
		+ |z(s)e^{Ks}|^{2}-\mu |y(s)e^{Ks}|^{2}-\frac{1}{\mu}|g(s,0,0,0,0,i_{0})e^{Ks}|^{2}\bigg]ds\bigg \}+C.
	\end{aligned}$
	\\The last inequality is justified by  $ 2ab \leq \mu | a| 
2 + (1/\mu )| b| 
2$. Taking the limit as  $T \rightarrow +\infty $
	on both sides of the inequality and invoking Lemma 4.1, we obtain
	\\$\begin{aligned}
		&E\bigg\{|y(t_{0})e^{Kt_{0}}|^{2}+\int_{t_{0}}^{+\infty}\bigg[(2K-2L-\mu)|y(s)e^{Ks}|^{2}+(1-2L)|z(s)e^{Ks}|^{2}\\&\quad
		+2L|f(s)e^{Ks}|^{2}+2L|r(s)e^{Ks}|^{2}\bigg] d s\bigg\} 
		\le E\left \{\int_{t_{0}}^{+\infty} \frac{1}{\mu}|g(s,0,0,0,0,\alpha(s))e^{Ks}|^{2} \right \}+C.
	\end{aligned}$
	
	A similar method can be employed to prove the estimate (\ref{ye2}) as was done for (\ref{xe2}).
\end{proof}

\begin{theorem}
	Let Assumption B be satisfied. When $K > 0$, the BSDE (\ref{yequ}) admits a unique solution $(y,z,r,f) \in L_{\mathbb{F}}^{2, K}\left(0, +\infty ; \mathbb{R}^{n}\right) \times L_{\mathbb{F}}^{2, K}\left(0, +\infty ; \mathbb{R}^{nd}\right) \times L_{\mathbb{F}}^{2, K}(0, +\infty ;$\\ $\mathbb{R}^{n}) \cap L^{2,H}(0,+\infty;\mathbb{R}^{n}) \times L_{\mathbb{F}}^{2, K}\left(0, +\infty ; \mathbb{R}^{n}\right)$.
\end{theorem}
\begin{proof}
	  The estimate (\ref{ye2}) ensures uniqueness, so it is sufficient to demonstrate existence. Firstly, for $n=1,2,\cdots$, we define $g _{n}(t)=1_{[0,n]}(t)g(t),  t\in [0,+\infty)$.  For each $ n $, define $(\bar{y}_{n},\bar{z}_{n},\bar{r}_{n},\bar{f}_{n})$ as the unique adapted solution to the following finite horizon backward SDE:
	 
	 $\begin{aligned}
	 &\bar{y}_{n}(t)=  \int_{t}^{n}g_{n}\left(s, \bar{y}_{n}(s), \bar{z}_{n}(s), \bar{r}_{n}(s), \bar{f}_{n}(s), \alpha(s)\right) d s+\int_{t}^{n} \bar{z}_{n}(s) d W(s) \\& \quad \quad \quad \quad \quad \quad \quad \quad \quad \quad \quad \quad \quad
	  +\int_{t}^{n} \bar{r}_{n}(s)dB^{H}(s)+\int_{t}^{n} \bar{f}_{n}(s)\cdot d \mathcal{M}(s), \quad t \in[0, n] .
	 \end{aligned}$
	 
	 To proceed, we define
	 \\ $\left(y_{n}(t), z_{n}(t), r_{n}(t),f_{n}(t)\right):=\left\{\begin{array}{ll}
	 \left(\bar{y}_{n}(t), \bar{z}_{n}(t), \bar{r}_{n}(t),\bar{f}_{n}(t) \right), & t \in[0, n], \\
	 (0,0,0,0), & t \in(n, +\infty) ,
	 \end{array}\right.$
	 \\ $g_{n}\left(\cdot, y_{n}(s), z_{n}(s), r_{n}(s),f_{n}(s),\cdot\right):=\left\{\begin{array}{ll}
	 g_{n}\left(\cdot, \bar{y}_{n}(s), \bar{z}_{n}(s), \bar{r}_{n}(s), \bar{f}_{n}(s), \cdot\right), & t \in[0, n], \\
	 0, & t \in(n, +\infty) .
	 \end{array}\right.$
	\\ From this, we deduce that $(y_{n}, z_{n}, r_{n},f_{n})\in L_{\mathbb{F}}^{2, K}\left(0, +\infty ; \mathbb{R}^{n}\right) \times L_{\mathbb{F}}^{2, K}\left(0, +\infty ; \mathbb{R}^{nd}\right) \times L_{\mathbb{F}}^{2, K}\left(0, +\infty ; \mathbb{R}^{n}\right) \cap L^{2,H}(0,+\infty;\mathbb{R}^{n}) \times L_{\mathbb{F}}^{2, K}\left(0, +\infty ; \mathbb{R}^{n}\right)$.
	And $(y_{n},z_{n},r_{n},f_{n})$ satisfies equation below:
	 \\ $\begin{aligned}
	& y_{n}(t)=  \int_{t}^{+\infty}g_{n}\left(s, y_{n}(s), z_{n}(s), r_{n}(s),f_{n}(s),\alpha(s)\right) d s+\int_{t}^{+\infty} z_{n}(s) d W(s) \\& \quad \quad \quad \quad \quad \quad \quad \quad  \quad \quad \quad \quad
	  +\int_{t}^{+\infty}r_{n}(s)dB^{H}(s)+\int_{t}^{+\infty}f_{n}(s)\cdot d\mathcal{M}(s), \quad t \in[0, +\infty) .
	 \end{aligned}$
	 
	Let us now assume that the following condition is true
	 \begin{equation}
	 	\begin{aligned}
	 	&\int_{0}^{+\infty} \left \langle y,g(s,y(s),z(s),r(s),f(s),i_{0})-g(s,0,0,0,0,i_{0})  \right \rangle e^{2Ks} ds\\&\ge \int_{0}^{+\infty}\bigg[L(-|y(s)|^{2}-|z(s)|^{2}+H(2H-1)r(s)\int_{0}^{+\infty}r(u)|s-u|^{2H-2}du+|f(s)|^{2})\bigg]ds.
	 	\end{aligned}\nonumber
	 \end{equation}
	 Building upon the previous assumption, and applying a similar derivation technique as in Lemma 4.2 the subsequent estimate is obtained:
	 \begin{equation}
	 \begin{aligned}
	 &E\bigg\{|(y(t_{0})-\bar{y}(t_{0}))e^{Kt_{0}}|^{2}+\int_{t_{0}}^{+\infty}\bigg[(2K-2L-\mu)|(y(s)-\bar{y}(s))e^{Ks}|^{2}\\&\quad+(1-2L)|(z(s)-\bar{z}(s))e^{Ks}|^{2}
	 +2L|(f(s)-\bar{f}(s))|^{2}\bigg] d s\\&\quad+\int_{t_{0}}^{+\infty}\int_{t_{0}}^{+\infty}2LH(2H-1)|(r(s)-\bar{r}(s))||(r(u)-\bar{r}(u))| |s-u|^{2H-2}duds\bigg\} \\&
	 \le E\left \{\int_{t_{0}}^{+\infty} \frac{1}{\mu}|(g(s,\bar{y}(s),\bar{z}(s),\bar{r}(s),\bar{f}(s),\alpha(s))-\bar{g}(s,\bar{y}(s),\bar{z}(s),\bar{r}(s),\bar{f}(s),\alpha(s))e^{Ks}|^{2} \right \}.\label{exag}
	 \end{aligned}
	 \end{equation}
	 
	  The estimate (\ref{ye2}) reveals that $\left \{(y_{n},z_{n},r_{n},f_{n})  \right \}_{n=1}^{+\infty} $ constitutes a Cauchy sequence in $L_{\mathbb{F}}^{2, K}\left(0, +\infty ; \mathbb{R}^{n}\right)$, while (\ref{exag}) shows that $\left \{r_{n}  \right \}_{n=1}^{+\infty} $ is a Cauchy sequence in $L^{2,H}(0,+\infty;\mathbb{R}^{n})$. Therefore, limit of $\left \{(y_{n},z_{n},r_{n},f_{n})  \right \}_{n=1}^{+\infty} $ is denoted by $(y,z,r,f)$. When $K>0$, it deduces that
	  \\$\begin{aligned}
	  E\left[\int_{t}^{+\infty}\left(z_{n}(s)-z(s)\right) d W(s)\right]^{2} & =E \int_{t}^{+\infty}\left|z_{n}(s)-z(s)\right|^{2} d s \\&  
	   \leq E \int_{0}^{+\infty}\left|z_{n}(s)-z(s)\right|^{2} e^{K s} d s \rightarrow 0, \quad \text { as } n \rightarrow +\infty,
	  \end{aligned}$
	  \\It is established that $\int_{0}^{+\infty}z_{n}(s)dW(s)$ converges to $\int_{0}^{+\infty}z(s)dW(s)$ in $L_{\mathcal{F}_{t}}^{2}\left(\Omega ; \mathbb{R}^{n}\right)$.  The same reasoning ectends to shaow that $\int_{0}^{+\infty}r_{n}(s)dB^{H}(s)$ converges to $\int_{0}^{+\infty}r(s)dB^{H}(s)$ in $L_{\mathcal{F}_{t}}^{2}\left(\Omega ; \mathbb{R}^{n}\right)$, as a consequence of estimates (\ref{ye2}) and (\ref{exag}).  For any $K>0$, we can apply the H$\rm \mathop{o}\limits^{ \cdot \cdot}$lder inequality, yielding: 
	  \\  $\begin{aligned}
	  	&E\left[\int_{t}^{+\infty}(f_{n}(s)-f(s))\cdot d\mathcal{M}(s) \right] ^{2}=E\left[\int_{t}^{+\infty}(f_{n}(s)-f(s))q_{i_{0}j_{0}}\chi(\alpha(s-)=i_{0}) ds\right]^{2}
	  	\\& \leq E\left[\int_{0}^{+\infty}|(f_{n}(s)-f(s))q_{i_{0}j_{0}}\chi(\alpha(s-)=i_{0})|e^{\frac{K}{2} s} e^{-\frac{K}{2} s} ds\right]^{2}
	  	\\& \leq E\left[\left(\int_{0}^{+\infty}|(f_{n}(s)-f(s))q_{i_{0}j_{0}}\chi(\alpha(s-)=i_{0})|^{2}e^{K s}  ds\right)\left(\int_{0}^{+\infty} e^{-K s} d s\right)\right]
	  	\\& \leq CE\int_{0}^{+\infty}|f_{n}(s)-f(s)|^{2}e^{Ks}ds
	  	\rightarrow 0, \quad \text { as } n \rightarrow +\infty,
	  \end{aligned}$
	  \\and
\\	 $\begin{aligned}
	   &E\left[\int_{t}^{+\infty}\bigg(g_{n}\left(s, y_{n}(s), z_{n}(s), r_{n}(s),f_{n}(s),\alpha(s)\right)-g\left(s, y(s), z(s), r(s),f(s),\alpha(s)\right)\bigg) d s\right]^{2} \\&
	  \leq  E\bigg[\int_{0}^{+\infty}|g_{n}\left(s, y_{n}(s), z_{n}(s), r_{n}(s),f_{n}(s),\alpha(s)\right) \\&\quad -g\left(s, y(s), z(s), r(s),f(s),\alpha(s)\right)| e^{\frac{K}{2} s} e^{-\frac{K}{2} s} d s\bigg]^{2} \\&
	  \leq E\bigg[\bigg(\int_{0}^{+\infty} \mid g_{n}\left(s, y_{n}(s), z_{n}(s), r_{n}(s),f_{n}(s),\alpha(s)\right)\\&\quad -g(s, y(s), z(s), r(s),f(s),\alpha(s))|^{2} e^{K s} d s\bigg)\left(\int_{0}^{+\infty} e^{-K s} d s\right)\bigg] \\&
	  \leq  C E \int_{0}^{+\infty}\bigg[\left|y_{n}(s)-y(s)\right|^{2}+\left|z_{n}(s)-z(s)\right|^{2}+|r_{n}(s)-r(s)|^{2}+|f_{n}(s)-f(s)|^{2}\bigg] e^{K s} d s \\&
	   \rightarrow 0, \quad \text { as } n \rightarrow +\infty,
	  \end{aligned}$
	  \\thus $\int_{0}^{+\infty}f_{n}(s)\cdot d\mathcal{M}(s)$ converges to $\int_{0}^{+\infty}f(s)\cdot d\mathcal{M}(s)$ in $L_{\mathcal{F}_{t}}^{2}\left(\Omega ; \mathbb{R}^{n}\right)$, $\int_{0}^{+\infty}g_{n}(s, y_{n}(s),$ \\$ z_{n}(s), r_{n}(s),f_{n}(s),\alpha(s))ds$ converges to $\int_{0}^{+\infty}g\left(s, y(s), z(s), r(s),f(s),\alpha(s)\right)$ in $L_{\mathcal{F}_{t}}^{2}\left(\Omega ; \mathbb{R}^{n}\right)$.
	  
	  Since $\lim\limits_{n\rightarrow +\infty}E\left[\int_{0}^{+\infty}|y(_{n}t)-y(t)|^{2}\right]=0$, there exists a subsequence of $\left \{y_{n}  \right \} $ such that $\lim\limits_{n\rightarrow +\infty}E\left[\int_{0}^{+\infty}|y(_{n}t)-y(t)|^{2}\right]=0,\quad a.s.\quad t\in[0,+\infty)$.
\end{proof}

\textit{Remark}: (\ref{yequ}) is the adjoint equation to (\ref{Xequ}), and, as with (\ref{Xequ}), the existence and uniqueness of its solution must be verified first. However, due to the involvement of Markov chains and fractional Brownian motion in (\ref{yequ}), standard methods are not directly applicable. Therefore, in Lemma 4.2, we obtain an estimate for the solution and analyze its continuous dependence on the initial conditions by imposing conditions on the inner product of $g$ and $y$. In Theorem 4.3, we establish the convergence of the solution by using It$\rm\hat{o}$'s isometry and the H$\rm \mathop{o}\limits^{ \cdot \cdot}$lder inequality, ultimately proving the existence and uniqueness of the solution to (\ref{yequ}).

Sections 3 and 4 have established the existence and uniaueness of solutions to the SDEs (\ref{Xequ}) and (\ref{yequ}), respectively, We now turn our attention to determining whether a solution exists for the coupled system of these equations.
\section{Infinite horizon FBSDE}
In this section, we consider the following infinite horizon FBSDE:
\begin{equation}
	\left\{\begin{aligned}
		&d x(s)=b(s, \theta(s),\alpha(s) d s+\sum_{i=1}^{d} \sigma_{i}(s, \theta(s),\alpha(s) d W_{i}(s)+\gamma(s)dB^{H}(s), \\&
		d y(s)=g(s, \theta(s),\alpha(s) d s+\sum_{i=1}^{d} z_{i}(s) d W_{i}(s)+r(s)dB^{H}(s)+f(s)\cdot d\mathcal{M}(s), \\&
		x(t_{0})=\Psi(y(t_{0})),
	\end{aligned}\right.\label{XY}
\end{equation}
where $s \in[t_{0}, +\infty)$, $ \theta (\cdot ) = (x(\cdot )^{T} , y(\cdot )^{T} , z(\cdot )^{T},r(\cdot )^{T},f(\cdot )^{T} )^{T} $ with $ z(\cdot ) = (z_{1}(\cdot )^{T} ,\cdots ,\\ z_{d}(\cdot )^{T} )^{T} $, $ \Psi : \Omega \times \mathbb{R} 
^{n} \rightarrow \mathbb{R}^{ n} $ and $ \Gamma = (g^{T} , b^{T} , \sigma ^{T} ) : [t_{0}, +\infty ) \times \Omega \times \mathbb{R} ^{n+n+nd+n+n} \rightarrow \mathbb{R} ^{n+n+nd+n+n} $
with $ \sigma = (\sigma _{1}^{T} ,\cdots , \sigma_{ d}^{T} )^{T} $ are two given mappings. We begin by outlining the conditions that the coefficients of (\ref{XY}) must satisfy, as follows:

\textbf{Assumption C.} (i) For any $ y \in \mathbb{R}^{n} $, $ \Psi (y) $ is $\mathcal{F}_{t}$-measurable, and for any $ \theta \in \mathbb{R} ^{n+n+nd+n+n} $, $ \Gamma (\cdot , \theta ) $ is $ \mathbb{F} $-progressively measurable. Moreover, $ \Psi (0) \in L^{2}_{\mathcal{F}_{t}}(\Omega ; \mathbb{R}^{ n}) $ and there exists a constant $ K \in \mathbb{R} $ such that $ \Gamma (\cdot , 0) \in L
^{2,K}_{\mathbb{F}}(t, +\infty ; \mathbb{R}^{n+n+nd+n+n}) $.
\\(ii)The functions $\Psi$  and $ \Gamma $ are uniformly Lipschitz continuous with respect to $ y $ and $ \theta $ , respectively.
The Lipschitz constant of $ l $ with respect to $ k $ is
labeled as $ l_{lk} $, where $ l = \Psi  , g, b, \sigma$ and $ k = x, y, z , r, f$.
\\(iii)The functions $ g $ and $ b $ are monotonic with respect to $ y $ and $ x $: $\langle g(s, x, y, z,r,f,i_{0})-g(s, x, \bar{y}, z,r,f,i_{0}), \widehat{y}\rangle \geq-\kappa_{y}|\widehat{y}|^{2}$, $\langle b(s, x, y, z,r,f,i_{0})-b(s, \bar{x}, y, z,r,f,i_{0}), \widehat{x}\rangle \leq-\kappa_{x}|\widehat{x}|^{2}$, where $\widehat{x}=x-\bar{x},\widehat{y}=y-\bar{y}$, $ \kappa_{x}$ and $\kappa_{y}$ are constants, $i_{0}\in \mathbb{S}$.

For any  $ \xi \in L_{\mathcal{F}_{t}}^{2}\left(\Omega ; \mathbb{R}^{n}\right)$  and function  $ \rho(\cdot)=   \left(\varphi(\cdot)^{T}, \psi(\cdot)^{T}, \eta(\cdot)^{T}\right)^{T} \in L_{\mathbb{F}}^{2, K}(0, +\infty ;$\\$ \mathbb{R}^{n+n+n d+n+n}) $  with  $ \eta(\cdot)=\left(\eta_{1}(\cdot)^{T}, \cdots, \eta_{d}(\cdot)^{T}\right)^{T}  $, we introduce a family of FBSDEs parameterized by  $ \tau \in[0,1] $ prior to solving the FBSDE in (\ref{XY}):
\begin{equation}
	\left\{\begin{aligned}
	d x^{\tau}(s)=&\left[b^{\tau}\left(s, \theta^{\tau}(s),\alpha(s)\right)+\psi(s)\right] d s\\&+\sum_{i=1}^{d}\left[\sigma_{i}^{\tau}\left(s, \theta^{\tau}(s),\alpha(s)\right)+\eta_{i}(s)\right] d W_{i}(s)+\gamma^{\tau}(s)dB^{H}(s), \\
	d y^{\tau}(s)=&\left[g^{\tau}\left(s, \theta^{\tau}(s),\alpha(s)\right)+\varphi(s)\right] d s\\&+\sum_{i=1}^{d} z_{i}^{\tau}(s) d W_{i}(s)+r^{\tau}(s)dB^{H}(s)+f^{\tau}(s)\cdot d\mathcal{M}(s), \\
	x^{\tau}(t_{0})=&\Psi^{\tau}\left(y^{\tau}(t_{0})\right)+\xi,
	\end{aligned}\right.\label{tau}
\end{equation}
where $\theta^{\tau}(\cdot)=(x^{\tau}(\cdot)^{T},y^{\tau}(\cdot)^{T},z^{\tau}(\cdot)^{T},r^{\tau}(\cdot)^{T},f^{\tau}(\cdot)^{T})^T$ with $z^{\tau}(\cdot)=(z^{\tau}_{1}(\cdot)^{T},\cdots,z^{\tau}_{d}(\cdot)^{T})$, $\Gamma^{\tau}=((g^{\tau})^{T},(b^{\tau})^{T},(\sigma^{\tau})^{T})^{T}$ with $\sigma^{\tau}=((\sigma^{\tau}_{1})^{T},\cdots,(\sigma^{\tau}_{d})^{T})^{T}$, and $\Psi^{\tau}, g^{\tau}, b^{\tau}$ are defined by $
\Psi^{\tau}(y)=\tau \Psi(y)$, $g^{\tau}(s, \theta)=\tau g(s, \theta)-(1-\tau) \kappa_{y} y$, $b^{\tau}(s, \theta)=\tau b(s, \theta)-(1-\tau) \kappa_{x} x$, $\sigma^{\tau}(s, \theta)=\tau \sigma(s, \theta)$, respectively,
for any $(s,\omega,\theta)\in [t_{0},+\infty)\times\Omega\times\mathbb{R}^{n+n+nd+n+n}$.

Setting $\tau=0$ results in (\ref{tau}) reducing to 
\begin{equation}
\left\{\begin{aligned}
&d x^{0}(s)=\left[-\kappa_{x} x+\psi(s)\right] d s+\sum_{i=1}^{d}\eta_{i}(s) d W_{i}(s)+\gamma^{0}(s)dB^{H}(s), \\&
d y^{0}(s)=\left[-\kappa_{y} y+\varphi(s)\right] d s+\sum_{i=1}^{d} z_{i}^{0}(s) d W_{i}(s)+r^{0}(s)dB^{H}(s)+f^{0}(s)\cdot d\mathcal{M}(s), \\&
x^{0}(t_{0})=\xi.
\end{aligned}\right.\label{reduced}
\end{equation}
 lt is evident that (\ref{reduced}) is a decoupled equation, allowing us to solve one equation independently before addressing the other. Thus by Theorem 3.2 and Theorem 4.3, we derive the following result.
 
 \textit{Remark}: It can be observed that Assumption C only requires the fundamental conditions of boundedness, Lipschitz continuity, and monotonicity. In contrast, the paper \cite{wei2021infinite} adds extra domination-monotonicity conditions on top of these. This paper solves a more complex problem than \cite{wei2021infinite} under the classical conditions, without the need for any additional assumptions.
 
\begin{corollary}
	Let $\kappa_{x},\kappa_{y}\in\mathbb{R}$ and assume $\kappa_{x}>\kappa_{y}$. Then, for $\forall \, K\in (\kappa_{y}, \kappa_{x})$, $\forall \, \xi \in L_{\mathcal{F}_{t}}^{2}\left(\Omega ; \mathbb{R}^{n}\right)$, $\rho(\cdot)\in L_{\mathbb{F}}^{2, K}\left(0, +\infty ; \mathbb{R}^{n+n+n d+n+n}\right)$, (\ref{reduced}) admits a unique solution.
\end{corollary}
\begin{lemma}
	Let Asumption C holds. Let $\tau\in [0,1]$ and $(\xi,\rho(\cdot)),(\bar{\xi},\bar{\rho}(\cdot))\in L_{\mathcal{F}_{t}}^{2}(\Omega ;$\\$ \mathbb{R}^{n})\times L_{\mathbb{F}}^{2, K}\left(0, +\infty ; \mathbb{R}^{n+n+n d+n+n}\right)$. Suppose $\theta(\cdot),\bar{\theta}(\cdot)\in L_{\mathbb{F}}^{2, K}\left(0, +\infty ; \mathbb{R}^{n+n+n d+n+n}\right)$ are solutions to (\ref{tau}) with $(\Psi^{\tau}+\xi,\Gamma^{\tau}+\rho)$ and $(\Psi^{\tau}+\bar{\xi},\Gamma^{\tau}+\bar{\rho})$, respectively. Then there exists a constant $C>0$ depending on the Lipschitz constants, $\kappa_{x},\kappa_{y}$ such that
	\begin{eqnarray}
		\begin{aligned}
		&E\left\{\left|(y(t_{0})-\bar{y}(t_{0})) e^{K t_{0}}\right|^{2}+\int_{t_{0}}^{+\infty}\left|(\theta(s)-\bar{\theta}(s)) e^{K s}\right|^{2} d s\right\} \\&\quad 
		\leq C E\left\{\left|(\xi-\bar{\xi}) e^{K t_{0}}\right|^{2}+\int_{t_{0}}^{+\infty}\left|(\rho(s)-\bar{\rho}(s)) e^{K s}\right|^{2} d s\right\}.
		\end{aligned}\label{lem5.2}
	\end{eqnarray}
	
\end{lemma}
\begin{proof}
	For convenience, we denote $\tilde{\xi}=(\xi-\bar{\xi})e^{Kt_{0}}$, $\tilde{l}(\cdot)=(l(\cdot)-\bar{l}(\cdot))e^{K\cdot}$, where $l(\cdot)=\theta(\cdot),x(\cdot),y(\cdot),z(\cdot),r(\cdot),f(\cdot),\rho(\cdot),\varphi(\cdot), \psi(\cdot), \eta(\cdot) $. For simplicity, $ s $ is abbreviated, except for the term 
	$e^{2Ks}$.
	
	This allows us to rewrite (\ref{xe2}) as follows: 
	\begin{equation}
		\begin{aligned}
		&\left(2 \kappa_{x}\right.  \left.-2 K-l_{\sigma x}^{2}-2 \mu\right) E \int_{t_{0}}^{+\infty}|\widetilde{x}|^{2} d s 
		\leq  E\bigg\{|\tau(\Psi(y(t_{0}))-\Psi(\bar{y}(t_{0}))) e^{K t_{0}}+\widetilde{\xi}|^{2} \\&\quad \quad
		 +\int_{t_{0}}^{+\infty}\bigg[\frac{1}{\mu}|\tau(b(\bar{x}, y, z,r,f,\alpha)-b(\bar{\theta},\alpha)) e^{K s}+\widetilde{\psi}|^{2} 
		 \\&\quad \quad+\left(1+\frac{l_{\sigma x}^{2}}{\mu}\right)\left|\tau(\sigma(\bar{x}, y, z,r,f,\alpha)-\sigma(\bar{\theta},\alpha)) e^{K s}+\widetilde{\eta}\right|^{2}\bigg] d s\bigg\}.
		\end{aligned}\label{5.4}
	\end{equation}
		In a like fashion, the estimate (\ref{ye2}) can be reformulated as:
	\begin{equation}
		\begin{aligned}
			&E\left\{|\widetilde{y}(t_{0})|^{2}+\int_{t_{0}}^{+\infty}\left[\left(2 K-2L- \mu\right)|\widetilde{y}|^{2}+(1-2L)|\widetilde{z}|^{2}+2L|\widetilde{f}|^{2}+2L|\widetilde{r}|^{2}\right] d s\right\} \\&
			\quad \leq \frac{1}{\mu} E \int_{t_{0}}^{+\infty}\left|\tau(g(x, \bar{y}, \bar{z},\bar{f},\bar{r},\alpha)-g(\bar{\theta},\alpha)) e^{K s}+\widetilde{\varphi}\right|^{2} d s.
		\end{aligned}\label{5.5}
	\end{equation}
	In line with Assumption C, we can evaluate the term on the right-hand side of (\ref{5.4}).
	
	$\begin{aligned}
	&\left(2 \kappa_{x}-2 K-l_{\sigma x}^{2}-2 \mu\right) E \int_{t_{0}}^{+\infty}|\widetilde{x}|^{2} d s \leq E\bigg\{2\left(l_{\Psi y}^{2}|\widetilde{y}(t_{0})|^{2}+|\widetilde{\xi}|^{2}\right) \\&
	\quad+\int_{t_{0}}^{+\infty}\bigg[\frac{5}{\mu}\left(l_{b y}^{2}|\widetilde{y}|^{2}+l_{b z}^{2}|\widetilde{z}|^{2}+l_{b r}^{2}|\widetilde{r}|^{2}+l_{b f}^{2}|\widetilde{f}|^{2}+|\widetilde{\psi}|^{2}\right) \\&
	\quad+5\left(1+\frac{l_{\sigma x}^{2}}{\mu}\right)\left(l_{\sigma y}^{2}|\widetilde{y}|^{2}+l_{\sigma_{z}}^{2}|\widetilde{z}|^{2}+l_{\sigma_{r}}^{2}|\widetilde{r}|^{2}+l_{\sigma_{f}}^{2}|\widetilde{f}|^{2}+|\widetilde{\eta}|^{2}\right)\bigg] d s\bigg\} .
	\end{aligned}$
	\\	In light of the assumptions in Lemma 3.4, we can conclude that $\mu_{1}=\kappa_{x}-K-(l_{\sigma x}^{2}/2)>0$. If we restrict $\mu\in (0,\mu_{1})$ and denote $\pi=\frac{1}{2 \kappa_{x}-2 K-l_{\sigma x}^{2}-2 \mu}$, then (\ref{5.4}) is reduced to
	\begin{eqnarray}
	E \int_{t_{0}}^{+\infty}|\widetilde{x}|^{2} d s &&\leq  E\bigg\{2l_{\Psi y}^{2}\pi|\widetilde{y}(t_{0})|^{2}+2\pi|\widetilde{\xi}|^{2}+\int_{t_{0}}^{+\infty}\bigg[5\pi\left(\frac{l_{b y}^{2}+l_{\sigma  y}^{2}(\mu+l_{\sigma x}^{2}) }{\mu}\right)|\widetilde{y}|^{2}\nonumber \\&&+5\pi\left(\frac{l_{b z}^{2}+l_{\sigma  z}^{2}(\mu+l_{\sigma x}^{2}) }{\mu}\right)|\widetilde{z}|^{2}+5\pi\left(\frac{l_{b r}^{2}+l_{\sigma  r}^{2}(\mu+l_{\sigma x}^{2}) }{\mu}\right)|\widetilde{r}|^{2}\nonumber \\&&+5\pi\left(\frac{l_{b f}^{2}+l_{\sigma  f}^{2}(\mu+l_{\sigma x}^{2}) }{\mu}\right)|\widetilde{f}|^{2}+\frac{5\pi}{\mu}|\widetilde{\psi}|^{2}+5\pi\left(1+\frac{l_{\sigma x}^{2}}{\mu}\right)\left|\widetilde{\eta}\right|^{2}\bigg] d s\bigg\}.	\label{5.6}
	\end{eqnarray}
	
	We now focus on approximating the term on the right-hand side of (\ref{5.5}). 
	\begin{eqnarray}
		\begin{aligned}
		&E\left\{|\widetilde{y}(t_{0})|^{2}+\int_{t_{0}}^{+\infty}\left[\left(2 K-2L- \mu\right)|\widetilde{y}|^{2}+(1-2L)|\widetilde{z}|^{2}+2L|\widetilde{f}|^{2}+2L|\widetilde{r}|^{2}\right] d s\right\} \\&
		\quad \leq \frac{1}{\mu} E \int_{t_{0}}^{+\infty}\left(\tau l_{g x}|\widetilde{x}|+|\widetilde{\varphi}|\right)^{2} d s \leq \frac{2}{\mu} E \int_{t_{0}}^{+\infty}\left(l_{g x}^{2}|\widetilde{x}|^{2}+|\widetilde{\varphi}|^{2}\right) d s.
		\end{aligned}\label{5.7}
	\end{eqnarray}
		The assumptions of Lemma 4.2 imply that $\mu_{2}=2K-2L>0$, then we restrict $\mu\in (0,\mu_{2})$. Assume $0<\mu<\min\{\mu_{1},\mu_{2}\}$, and substituting (\ref{5.6}) into (\ref{5.7}) yields

$\begin{aligned}
&E\left\{|\widetilde{y}(t_{0})|^{2}+\int_{t_{0}}^{+\infty}\bigg[\left(2 K-2L- \mu\right)|\widetilde{y}|^{2}+(1-2L)|\widetilde{z}|^{2}+2L|\widetilde{f}|^{2}+2L|\widetilde{r}|^{2}\bigg] d s\right\} \\& \leq E \int_{t_{0}}^{+\infty}\frac{2}{\mu}|\widetilde{\varphi}|^{2} d s
+ \frac{2l^{2}_{gx}}{\mu}E\bigg[2l_{\Psi y}^{2}\pi|\widetilde{y}(t_{0})|^{2}+2\pi|\widetilde{\xi}|^{2}\bigg]
\\&\quad+\frac{2l^{2}_{gx}5\pi}{\mu^{2}}E\int_{t_{0}}^{+\infty}\bigg[(l_{b y}^{2}+l_{\sigma  y}^{2}(\mu+l_{\sigma x}^{2}))|\widetilde{y}|^{2}+(l_{b z}^{2}+l_{\sigma  z}^{2}(\mu+l_{\sigma x}^{2}))|\widetilde{z}|^{2}\\&\quad+(l_{b r}^{2}+l_{\sigma  r}^{2}(\mu+l_{\sigma x}^{2}))|\widetilde{r}|^{2}+(l_{b f}^{2}+l_{\sigma  f}^{2}(\mu+l_{\sigma x}^{2}))|\widetilde{f}|^{2}+|\widetilde{\psi}|^{2}+(\mu+l_{\sigma x}^{2})|\widetilde{\eta}|^{2}\bigg] d s.
\end{aligned}$
\\Inserting (\ref{5.6}) into the above inequality leads to 
\begin{eqnarray}
	\begin{aligned}
		&E\left\{|\widetilde{y}(t_{0})|^{2}+\int_{t_{0}}^{+\infty}\bigg[|\widetilde{x}|^{2}+\left(2 K-2L- \mu\right)|\widetilde{y}|^{2}+(1-2L)|\widetilde{z}|^{2}+2L|\widetilde{f}|^{2}+2L|\widetilde{r}|^{2}\bigg] d s\right\} \\& \leq E\bigg[2l_{\Psi y}^{2}\pi\left(\frac{2l^{2}_{gx}}{\mu}+1\right)|\widetilde{y}(t_{0})|^{2}+2\pi\left(\frac{2l^{2}_{gx}}{\mu}+1\right)|\widetilde{\xi}|^{2}\bigg]
		\\&\quad+E\int_{t_{0}}^{+\infty}\bigg[5\pi\left(\frac{l_{b y}^{2}+l_{\sigma  y}^{2}(\mu+l_{\sigma x}^{2}) }{\mu}\right)\left(\frac{2l^{2}_{gx}}{\mu}+1\right)|\widetilde{y}|^{2}\\&\quad+5\pi\left(\frac{l_{b z}^{2}+l_{\sigma  z}^{2}(\mu+l_{\sigma x}^{2}) }{\mu}\right)\left(\frac{2l^{2}_{gx}}{\mu}+1\right)|\widetilde{z}|^{2}\\&\quad+5\pi\left(\frac{l_{b r}^{2}+l_{\sigma  r}^{2}(\mu+l_{\sigma x}^{2}) }{\mu}\right)\left(\frac{2l^{2}_{gx}}{\mu}+1\right)|\widetilde{r}|^{2}\\&\quad+5\pi\left(\frac{l_{b f}^{2}+l_{\sigma  f}^{2}(\mu+l_{\sigma x}^{2}) }{\mu}\right)\left(\frac{2l^{2}_{gx}}{\mu}+1\right)|\widetilde{f}|^{2}\bigg] d s
		\\&\quad+E\bigg[ \int_{t_{0}}^{+\infty}\frac{2}{\mu}|\widetilde{\varphi}|^{2} d s
		+\frac{5\pi}{\mu} \left(\frac{2l^{2}_{gx}}{\mu}+1\right)|\widetilde{\psi}|^{2}\\&\quad+5\pi\left(1+\frac{l^{2}_{\sigma x}}{\mu}\right)\left(\frac{2l^{2}_{gx}}{\mu}+1\right)|\widetilde{\eta}|^{2}\bigg].
	\end{aligned}\label{5.8}
\end{eqnarray}
We now reformulate (\ref{5.8}) in the form of
\\$\begin{aligned}
&E\bigg\{\left(1-2l_{\Psi y}^{2}\pi\left(\frac{2l^{2}_{gx}}{\mu}+1\right)\right)|\widetilde{y}(t_{0})|^{2}\\&\quad+\int_{t_{0}}^{+\infty}\bigg[|\widetilde{x}|^{2}+\left(1-5\pi\left(\frac{l_{b y}^{2}+l_{\sigma  y}^{2}(\mu+l_{\sigma x}^{2}) }{\mu}\right)(\frac{2l^{2}_{gx}}{\mu}+1)\right)|\widetilde{y}|^{2}
\\&\quad+\left(1-5\pi\left(\frac{l_{b z}^{2}+l_{\sigma  z}^{2}(\mu+l_{\sigma x}^{2}) }{\mu}\right)\left(\frac{2l^{2}_{gx}}{\mu}+1\right)\right)|\widetilde{z}|^{2} \\&\quad+\left(1-5\pi\left(\frac{l_{b r}^{2}+l_{\sigma  r}^{2}(\mu+l_{\sigma x}^{2}) }{\mu}\right)\left(\frac{2l^{2}_{gx}}{\mu}+1\right)\right)|\widetilde{r}|^{2}\\&\quad+\left(1-5\pi\left(\frac{l_{b f}^{2}+l_{\sigma  f}^{2}(\mu+l_{\sigma x}^{2}) }{\mu}\right)\left(\frac{2l^{2}_{gx}}{\mu}+1\right)\right)|\widetilde{f}|^{2}\bigg] d s\bigg\}
\\& \leq  E\bigg\{2\pi\left(\frac{2l^{2}_{gx}}{\mu}+1\right)|\widetilde{\xi}|^{2}+\int_{t_{0}}^{+\infty}\frac{2}{\mu}|\widetilde{\varphi}|^{2} d s
+\frac{5\pi}{\mu} \left(\frac{2l^{2}_{gx}}{\mu}+1\right)|\widetilde{\psi}|^{2}\\&\quad+5\pi\left(1+\frac{l^{2}_{\sigma x}}{\mu}\right)\left(\frac{2l^{2}_{gx}}{\mu}+1\right)|\widetilde{\eta}|^{2}\bigg\}.
\end{aligned}$
\\Let $\pi$ be small enough so that the coefficients on the left side of the above inequality are all greater than zero, while the minimum and maximum values of the coefficients on the right-hand side of the above inequality are denoted as $C_{1}$ and $ C_{2} $, respectively.
\begin{eqnarray}
	\begin{aligned}
	E\left\{|\widetilde{y}(t_{0})|^{2}+\int_{t_{0}}^{+\infty}|\widetilde{\theta}|^{2} d s\right\}\leq \frac{C_{2}}{C_{1}} E\left\{|\widetilde{\xi}|^{2}+\int_{t_{0}}^{+\infty}|\widetilde{\rho}|^{2} d s\right\}.
	\end{aligned}
\end{eqnarray}
\end{proof}

\textit{Remark}: The inclusion of supplementary conditions in \cite{wei2021infinite} necessitates a detailed, case-by-case analysis to establish (\ref{lem5.2}). In comparison, this paper proves (\ref{lem5.2}) straightforwardly using the fundamental Assumption C, reducing the effort required.

\begin{lemma}
	Let Assumption C hold for the given coefficients $(\Psi,\Gamma)$. Then, there exists a constant $\delta_{0}>0$ independent of $\tau$, such that for some $\tau_{0}\in[0,1)$, (\ref{tau}) possesses a unique solution for any $(\xi,\rho(\cdot))$. Additionally, the same conclusion is also true for $\tau=\tau_{0}+\delta$ with $\delta\in [0,\delta_{0}]$ and $\tau \le 1$.
\end{lemma}
\begin{proof}
	Let $\delta\in [0,\delta_{0}]$, $\xi \in L_{\mathcal{F}_{t}}^{2}\left(\Omega ; \mathbb{R}^{n}\right)$, $\rho(\cdot) \in  L_{\mathbb{F}}^{2, K}\left(0, +\infty ; \mathbb{R}^{n+n+n d+n+n}\right)$. For any $p \in L_{\mathcal{F}_{t}}^{2}\left(\Omega ; \mathbb{R}^{n}\right)$ and $\theta(\cdot) \in  L_{\mathbb{F}}^{2, K}\left(0, +\infty ; \mathbb{R}^{n+n+n d+n+n}\right)$, We specify the equation as follows:
	\begin{eqnarray}
		\left\{\begin{aligned}
		&d X(s)=\left[b^{\tau_{0}}(s, \Theta(s),\alpha(s)+\psi^{\prime}(s)\right] d s\\&\quad \quad \quad \quad+\sum_{i=1}^{d}\left[\sigma_{i}^{\tau_{0}}(s, \Theta(s),\alpha(s)+\eta_{i}^{\prime}(s)\right] d W_{i}(s)+\gamma^{\tau_{0}}(s)dB^{H}, \\&
		d Y(s)=\left[g^{\tau_{0}}(s, \Theta(s))+\varphi^{\prime}(s)\right] d s\\&\quad \quad \quad \quad+\sum_{i=1}^{d} Z_{i}(s) d W_{i}(s)+\hat{r}(s)dB^{H}+\hat{f}(s)\cdot d\mathcal{M}(s),\\&
		X(t_{0})=\Psi^{\tau_{0}}(Y(t_{0}))+\xi^{\prime},
		\end{aligned}\right. \label{The}
	\end{eqnarray}
	where $\xi^{\prime}=\delta \Psi(p)+\xi$, $\varphi^{\prime}(\cdot)=\delta g(\cdot, \theta(\cdot))+\delta \kappa_{y} y(\cdot)+\varphi(\cdot)$, $\psi^{\prime}(\cdot)=\delta b(\cdot, \theta(\cdot))+\delta \kappa_{x} x(\cdot)+\psi(\cdot)$, $\eta^{\prime}(\cdot)=\delta \sigma(\cdot, \theta(\cdot))+\eta(\cdot)$ with $s \in[0, +\infty)$.
	
	We can claim that $\xi^{'} \in L_{\mathcal{F}_{t}}^{2}\left(\Omega ; \mathbb{R}^{n}\right)$, $\rho^{'}(\cdot)=(\varphi^{'}(\cdot)^{T},\psi^{'}(\cdot)^{T},\eta^{'}(\cdot)^{T})^{T} \in \\ L_{\mathbb{F}}^{2, K}\left(0, +\infty ; \mathbb{R}^{n+n+n d+n+n}\right)$. In this case, (\ref{The}) is uniquely solvable. 
	In fact, we denote $\Theta(\cdot)=(X(\cdot)^{T},Y(\cdot)^{T},Z(\cdot)^{T},\hat{r}(\cdot),\hat{f}(\cdot))^{T}$. Due to the arbitrariness of $p$ and $\theta(\cdot)$, we can interpret (\ref{The}) as specifying a mapping $\mathcal{T}_{\tau_{0}+\delta}$ from $L_{\mathcal{F}_{t}}^{2}\left(\Omega ; \mathbb{R}^{n}\right)\times\\ L_{\mathbb{F}}^{2, K}\left(0, +\infty ; \mathbb{R}^{n+n+n d+n+n}\right)$ into itself: $(Y(t_{0}), \Theta(\cdot))=\mathcal{T}_{\tau_{0}+\delta}(p, \theta(\cdot))$.
	For any $p,\bar{p} \in L_{\mathcal{F}_{t}}^{2}\left(\Omega ; \mathbb{R}^{n}\right)$ and $\theta(\cdot),\bar{\theta}(\cdot) \in  L_{\mathbb{F}}^{2, K}\left(0, +\infty ; \mathbb{R}^{n+n+n d+n+n}\right)$, let $(Y(t_{0}), \Theta(\cdot))=\mathcal{T}_{\tau_{0}+\delta}(p, \theta(\cdot))$ and $(\bar{Y}(t_{0}), \bar{\Theta}(\cdot))=\mathcal{T}_{\tau_{0}+\delta}(\bar{p}, \bar{\theta}(\cdot))$. We denote $\tilde{p}=(p-\bar{p})e^{Kt}$, $\tilde{l}(\cdot)=(l(\cdot)-\bar{l}(\cdot))e^{K\cdot}$, where $l(\cdot)=\theta(\cdot),x(\cdot),y(\cdot),z(\cdot),r(\cdot),f(\cdot),\Theta(\cdot),X(\cdot), Y(\cdot),Z(\cdot), \hat{r}(\cdot),\hat{f}(\cdot) $. To streamline the notation, $ s $ is abbreviated, with the exception of the term
	$e^{2Ks}$. From Lemma 5.2, we conclude 
	\\$\begin{aligned}
	&E\left\{|\tilde{Y}(t_{0})|^{2}+\int_{t_{0}}^{+\infty}|\widetilde{\Theta}|^{2} d s\right\} \leq C \delta^{2} E\bigg\{\left|(\Psi(p)-\Psi(\bar{p})) e^{K t}\right|^{2} 
	 +\int_{t_{0}}^{+\infty}\bigg[|(g(\theta)-g(\bar{\theta})) e^{K s}\\&\quad \quad+\kappa_{y} \widetilde{y}|^{2} 
	 +\left|(b(\theta)-b(\bar{\theta})) e^{K s}+\kappa_{x} \widetilde{x}\right|^{2} 
	 +\left|(\sigma(\theta)-\sigma(\bar{\theta})) e^{K s}\right|^{2}\bigg] d s\bigg\} .
	\end{aligned}$
	\\Then there exists a constant $C_{5}$ that is independent of $\alpha_{0}$ and $\delta$, such that 
	
	$$E\left\{|\tilde{Y}(t_{0})|^{2}  +\int_{t_{0}}^{+\infty}|\widetilde{\Theta}|^{2} d s\right\} \leq C_{5} \delta^{2} E\left\{|\tilde{p}|^{2}  +\int_{t_{0}}^{+\infty}|\tilde{\theta}|^{2} d s\right\}.$$
	By selecting $\delta_{0}=1/(2\sqrt{C_{5}})$, it follows that for $\delta\in [0,\delta_{0}]$, the mapping $\mathcal{T}_{\tau_{0}+\delta}$ is contractive. As a consequence, the contracting mapping principle guarantees that equation (\ref{The}) also has a unique solution. 
	Then according the extension theorem, (\ref{tau}) also has a unique solution.
\end{proof}

\begin{theorem}
	Let the coefficients $(\Psi,\Gamma)$ satisfy Assumption C. (\ref{XY}) admits a solution $\theta(\cdot)$. 
	Let $(\bar{\Psi},\bar{\Gamma})$ be another coefficients and $\bar{\theta}(\cdot)\in  L_{\mathbb{F}}^{2, K}\left(0, +\infty ; \mathbb{R}^{n+n+n d+n+n}\right)$ be a solution to the (\ref{XY}) with $(\bar{\Psi},\bar{\Gamma})$. Assume that $\bar{\Psi}(\bar{y}(t_{0}))\in L_{\mathcal{F}_{t}}^{2}\left(\Omega ; \mathbb{R}^{n}\right)$ and $\bar{\Gamma}(\cdot,\bar{\theta}(\cdot))\in  L_{\mathbb{F}}^{2, K}\left(0, +\infty ; \mathbb{R}^{n+n+n d+n+n}\right)$. Then,
	\begin{eqnarray}
	\quad\\\label{5.13}
		\begin{aligned}
		&E\left\{\left|(y(t_{0})-\bar{y}(t_{0})) e^{K t_{0}}\right|^{2}+\int_{t_{0}}^{+\infty}\left|(\theta(s)-\bar{\theta}(s)) e^{K s}\right|^{2} d s\right\} \\&
		\leq C E\left\{\left|(\Psi(\bar{y}(t_{0}))-\bar{\Psi}(\bar{y}(t_{0}))) e^{K t_{0}}\right|^{2}+\int_{t_{0}}^{+\infty}\left|(\Gamma(s, \bar{\theta}(s))-\bar{\Gamma}(s, \bar{\theta}(s))) e^{K s}\right|^{2} d s\right\},
		\end{aligned}\nonumber
	\end{eqnarray}
	where $C$ is a constant.
\end{theorem}
\begin{proof}
	When $ \tau $ = 0, Corollary 5.1 shows
	the unique solvability of FBSDE (\ref{tau}). Lemma 5.3 further infers that FBSDE (\ref{tau})
	is uniquely solvable for any $ (\xi , \rho (\cdot )) $ and any $ \tau \in [0, 1] $. In particular, when $ \tau = 1 $ and
	$ (\xi , \rho (\cdot )) = (0, 0) $, (\ref{tau}) coincides with the original (\ref{XY}), thereby establishing the unique solvability of (\ref{XY}).
	
	Let $\tau=0,(\xi , \rho (\cdot )) = (0, 0)$ and
	$(\bar{\xi}, \bar{\rho}(\cdot))=(\bar{\Psi}(\bar{y}(t_{0}))-\Psi(\bar{y}(t_{0})), \bar{\Gamma}(\cdot, \bar{\theta}(\cdot))-\Gamma(\cdot, \bar{\theta}(\cdot)))$, the estimate (\ref{5.13}) is followed from (\ref{lem5.2}) in Lemma 5.2.
\end{proof}

\textbf{Example}: Let $n=d=1$. Consider the following FBSDE:
\begin{eqnarray}
	\left\{\begin{array}{l}
	\mathrm{d} x(s)=b(s,\theta(s),\alpha(s)) \mathrm{d} s+\sigma(s,\theta(s),\alpha(s)) \mathrm{d} W(s)+r(s)dB^{H}(s),  \\
	\mathrm{d} y(s)=g(s,\theta(s),\alpha(s))\mathrm{d} s+z(s) \mathrm{d} W(s)+r(s)dB^{H}(s)+f(s)\cdot d\mathcal{M}(s),  \\
	x(0)=\Psi(y(0),\alpha(s)),
	\end{array}\right.\nonumber
\end{eqnarray}
where $b(s,\theta(s),1)=-2 y(s)+\sin y(s), b(s,\theta(s),2)=3y(s)-\sin y(s),\\ \sigma(s,\theta(s),1)=-2 z(s)+\sin z(s), \sigma(s,\theta(s),2)=z(s)+\sin z(s), g(s,\theta(s),1)=-2 x(s)+\sin x(s)+y(s), g(s,\theta(s),2)=1/2 x(s)-\sin x(s)+y(s), \Psi(y(0),1)=-2 y(0)+\sin y(0), \\\Psi(y(0),2)=3 y(0)-\sin y(0) $ with $s \in[0, +\infty)$.
It is evident  that the correlation coefficients in the example satisfy Assumption C (i)-(iii), and the parameters $\kappa_{x}, \kappa_{y}, K$ meet $-1\le \kappa_{y}\le \kappa_{x} \le 0, K=\frac{\kappa_{x}+\kappa_{y}}{2}=-\frac{1}{2}$. As established by Theorem 5.1, the equation above possesses a unique solution.

The existence and uniqueness of solutions to the FBSDE have already been guaranteed in previous sections. With this in place, we now focus on the two-person zero-sum optimal control problem

\section{ Stochastic LQ problems with random coefficients}
This section is dedicated to the study of the LQ problem (\ref{xx})-(\ref{jj}). Initially, we outline the assumptions pertaining to the coefficients: 

$A(\cdot),C_{i}(\cdot),H(\cdot) \in L_{\mathbb{F}}^{+\infty}(0,+\infty;\mathbb{R}^{n\times n}), B_{1}(\cdot),B_{2}(\cdot),D_{1i}(\cdot),D_{2i}(\cdot) \in L_{\mathbb{F}}^{+\infty}(0,+\infty;\\\mathbb{R}^{n\times m}), i=1,\cdots,d$. To simplify the notation, we denote $C(\cdot)=(C_{1}(\cdot)^{T},C_{2}(\cdot)^{T},\cdots,\\C_{d}(\cdot)^{T}
)^{T}, D_{1}(\cdot)=(D_{11}(\cdot)^{T},D_{12}(\cdot)^{T},\cdots,D_{1d}(\cdot)^{T}
)^{T},D_{2}(\cdot)=(D_{21}(\cdot)^{T},D_{22}(\cdot)^{T},\cdots,\\D_{2d}(\cdot)^{T}
)^{T},B(\cdot)=(B_{1}(\cdot),B_{2}(\cdot)), D(\cdot)=(D_{1}(\cdot),D_{2}(\cdot))$.
$Q(\cdot) \in L_{\mathbb{F}}^{+\infty}(0,+\infty;\mathbb{S}^{ n}) ,S_{1}(\cdot),\\S_{2}(\cdot) \in L_{\mathbb{F}}^{+\infty}(0,+\infty;\mathbb{R}^{m \times n}),R_{11}(\cdot),R_{12}(\cdot),R_{21}(\cdot),R_{22}(\cdot) \in L_{\mathbb{F}}^{+\infty}(0,+\infty;\mathbb{S}^{ m}) $. We adopt the notation $S(\cdot)=\begin{pmatrix}
S_{1}(\cdot) \\
S_{2}(\cdot)
\end{pmatrix}$, $R(\cdot)=\begin{pmatrix}
R_{11}(\cdot)&R_{12}(\cdot) \\
R_{21}(\cdot)&R_{22}(\cdot)
\end{pmatrix}$.

\textbf{Assumption D.} $R(\cdot)>0, \begin{pmatrix}
Q(\cdot)& S_{1}(\cdot)^{T}\\
S_{1}(\cdot)& R_{11}(\cdot)
\end{pmatrix}\ge 0, \begin{pmatrix}
Q(\cdot)& S_{2}(\cdot)^{T}\\
S_{2}(\cdot)& R_{22}(\cdot)
\end{pmatrix}\le 0$.

Let $ \lambda_{\max}( A(s, \omega,i )+A(s, \omega,i )^{T}) $ 
signify the largest eigenvalue of symmetrical matrix $( A(s, \omega,i )+A(s, \omega,i )^{T}) $. Because $ A(\cdot) $ is bounded, it leads to
\begin{equation}
	\underset{(s, \omega,i) \in[t_{0}, +\infty) \times \Omega\times \mathbb{S}}{\operatorname{esssup}} \lambda_{\max }\left(A(s, \omega,i)+A(s, \omega,i)^{T}\right)<+\infty.\nonumber
\end{equation} 
Let $\kappa_{x}$ denote a real value for which
\begin{eqnarray}
	\kappa_{x} \leq-\frac{1}{2} \underset{(s, \omega,i) \in[t_{0}, +\infty) \times \Omega\times \mathbb{S}}{\operatorname{esssup}} \lambda_{\max }\left(A(s, \omega,i)+A(s, \omega,i)^{T}\right) .\label{kax}
\end{eqnarray}
This allows us to acquire
\begin{eqnarray}
	\begin{aligned}
	\langle A&(s, \omega,i) x, x\rangle  =\frac{1}{2}\left\langle\left(A(s, \omega,i)+A(s, \omega,i)^{T}\right) x, x\right\rangle\\& \leq \frac{1}{2} \lambda_{\max }\left(A(s, \omega,i)+A(s, \omega,i)^{T}\right)|x|^{2} \\
	& \leq \frac{1}{2} \underset{(s, \omega,i) \in[t_{0}, +\infty) \times \Omega\times \mathbb{S}}{\operatorname{esssup}} \lambda_{\max }\left(A(s, \omega,i)+A(s, \omega,i)^{T}\right)|x|^{2} \leq-\kappa_{x}|x|^{2}
	\end{aligned}\label{KK}
\end{eqnarray}
for $\forall \, x\in \mathbb{R}^{n}$ and almost all $(s,\omega,i)\in [t_{0},+\infty)\times \Omega \times\mathbb{S}$. Next, let $K$ be such that
\begin{eqnarray}
	K<\kappa_{x}-\frac{\|C(\cdot)\|_{L_{\mathrm{F}}^{+\infty}\left(t, +\infty ; \mathbb{R}^{(n d) \times n}\right)}^{2}}{2}.
\end{eqnarray}
Hence, it can be asserted that equation (\ref{xx}) has a unique solution $ x(\cdot) $ for any given $ u(\cdot) $. We define the admissible control set as $\mathcal{U}^{K}[t_{0},+\infty)=\mathcal{U}_{1}[t_{0},+\infty)\times \mathcal{U}_{2}[t_{0},+\infty)=L_{\mathbb{F}}^{2, K}\left(0, +\infty ; \mathbb{R}^{m}\right)$. A control $u(\cdot)\in \mathcal{U}^{K}[t_{0},+\infty)$ is considered admissible, where $u(\cdot)=(u_{1}(\cdot)^{T},u_{2}(\cdot)^{T})^{T}$. The state process $x(\cdot)=x(\cdot;t_{0},x_{t_{0}},u(\cdot))$ is  the admissible state process correspinding to $u(\cdot)$, and the pair $ (u(\cdot),x(\cdot)) $ is referred to as an admissible pair.

\textbf{Problem (LQ).} Assume that (\ref{kax}) and (\ref{KK}) are valid. For any $(t_{0},x_{t_{0}})\in [0,+\infty)\times L_{\mathcal{F}_{t}}^{2}\left(\Omega ; \mathbb{R}^{n}\right)$, determine a control $u^{*}(\cdot)$ such that 
\begin{eqnarray}
	J^{K}\left(t_{0}, x_{t_{0}} ; u^{*}(\cdot)\right)=\inf _{u(\cdot) \in \mathcal{U}^{K}[t_{0}, +\infty)} J^{K}\left(t_{0}, x_{t_{0}} ; u(\cdot)\right).
\end{eqnarray}
\begin{theorem}
	Assume that (\ref{kax}), (\ref{KK}), and Assumption D are fulfilled. Let $(t_{0},x_{t_{0}})\in [0,+\infty)\times L_{\mathcal{F}_{t}}^{2}\left(\Omega ; \mathbb{R}^{n}\right)$ represent the initial data, and let $(u^{*}(\cdot),x^{*}(\cdot))$ be an admissible pair, where $u^{*}(\cdot)=(u_{1}^{*}(\cdot)^{T},u_{2}^{*}(\cdot)^{T})^{T}\in \mathcal{U}_{1}[t_{0},+\infty)\times \mathcal{U}_{2}[t_{0},+\infty)$. Assume the following BSDE
	\begin{eqnarray}
	d y^{*}(s)= && -\big[\left(2 K I+A(s,\alpha(s))^{T}\right) y^{*}(s)+C(s,\alpha(s))^{T} z^{*}(s)+Q(s,\alpha(s)) x^{*}(s)\nonumber\\
	&& +S(s,\alpha(s))^{T} u^{*}(s)\big] d s +\sum_{i=1}^{d} z_{i}^{*}(s) d W_{i}(s)+r^{*}(s)dB^{H}(s)+f^{*}(s)\cdot d\mathcal{M}(s) \label{ady}
	\end{eqnarray}
	admits a solution $(y^{*}(\cdot),z^{*}(\cdot),r^{*}(\cdot),f^{*}(\cdot))\in L_{\mathbb{F}}^{2, K}\left(0, +\infty ; \mathbb{R}^{n}\right)\times L_{\mathbb{F}}^{2, K}\left(0, +\infty ; \mathbb{R}^{nd}\right)\times L_{\mathbb{F}}^{2, K}\left(0, +\infty ; \mathbb{R}^{n}\right)\times L_{\mathbb{F}}^{2, K}\left(0, +\infty ; \mathbb{R}^{n}\right)$. Then the following statements are equivalent:
	
	(i) $u^{*}(\cdot)$ is an optimal control at $(t_{0},x_{t_{0}})$;
	
	(ii) $u^{*}(\cdot)$ satisfies
	\begin{eqnarray}
		B(\cdot)^{T} y^{*}(\cdot)+D(\cdot)^{T} z^{*}(\cdot)+S(\cdot)^{T}x^{*}(\cdot) +R(\cdot)u^{*}(\cdot)=0,\quad a.s..\label{u**}
	\end{eqnarray}
\end{theorem}	
\begin{proof}
Let $(u^{*}(\cdot),x^{*}(\cdot))$ be an admissible pair, where $u^{*}(\cdot)=(u_{1}^{*}(\cdot)^{T},u_{2}^{*}(\cdot)^{T})^{T}\in \mathcal{U}_{1}[t_{0},+\infty)\times \mathcal{U}_{2}[t_{0},+\infty)$, $ \varepsilon \in \mathbb{R} $. By definition, for $\forall \, u_{1}(\cdot) \in \mathcal{U}_{1}[t_{0}, +\infty), \forall \, u_{2}(\cdot) \in \mathcal{U}_{2}[t_{0}, +\infty)$, $u^{*}(\cdot)$ serves as an optimal control if and only if the following conditions hold:
\begin{eqnarray}
	J^{K}\left(t_{0}, x_{t_{0}} ; u_{1}^{*}(\cdot), u_{2}^{*}(\cdot)\right) \le J^{K}\left(t_{0}, x_{t_{0}} ; u_{1}^{*}(\cdot)+\varepsilon u_{1}(\cdot), u_{2}^{*}(\cdot)\right) , \label{ju1}
\end{eqnarray}
\begin{eqnarray}
	J^{K}\left(t_{0}, x_{t_{0}} ; u_{1}^{*}(\cdot), u_{2}^{*}(\cdot)\right) \ge J^{K}\left(t_{0}, x_{t_{0}} ; u_{1}^{*}(\cdot), u_{2}^{*}(\cdot)+\varepsilon u_{2}(\cdot)\right).\label{ju2}
\end{eqnarray}

Let $\forall \, u_{1}(\cdot) \in \mathcal{U}_{1}[t_{0}, +\infty),\varepsilon \in \mathbb{R}$, $x^{\varepsilon}(\cdot)$ being the solution to the following equation on $[t_{0}, +\infty)$:
\begin{eqnarray}
	\left\{\begin{aligned} 
	&d x^{\varepsilon}(s)=\big\{A(s,\alpha(s)) x^{\varepsilon}(s)+B_{1}(s,\alpha(s))\big[u_{1}^{*}(s)+\varepsilon u_{1}(s)\big]+B_{2}(s,\alpha(s)) u_{2}^{*}(s)\big\} d s \\&
	+\sum_{i=1}^{d}\big\{C_{i}(s,\alpha(s)) x^{\varepsilon}(s)+D_{1i}(s,\alpha(s))\big[u_{1}^{*}(s)+\varepsilon u_{1}(s)\big]+D_{2i}(s,\alpha(s)) u_{2}^{*}(s)\big\} d W_{i}(s) \\&\quad \quad \quad \quad \quad \quad \quad \quad \quad \quad \quad \quad \quad \quad \quad \quad \quad \quad \quad \quad \quad \quad \quad \quad \quad \quad \quad \quad+H(s)dB^{H}(s),\\&
	x^{\varepsilon}(t_{0})=x_{t_{0}},\quad \alpha(t_{0})=i .
	\end{aligned}\right.\nonumber
\end{eqnarray}
As a result, $x_{1}(\cdot)=\frac{x^{\varepsilon}(\cdot)-x^{*}(\cdot)}{\varepsilon}$ solves the equation presented below:
\begin{eqnarray}
	\left\{\begin{aligned}
	&d x_{1}(s)= {\big[A(s,\alpha(s)) x_{1}(s)+B_{1}(s,\alpha(s)) u_{1}(s)\big] d s } \\&\quad \quad \quad  
	+\sum_{i=1}^{d}\left[C_{i}(s,\alpha(s)) x_{1}(s)+D_{1i}(s,\alpha(s)) u_{1}(s)\right] d W_{i}(s),\quad s \in[t_{0}, +\infty), \\&
	x_{1}(t_{0})=0,\quad\alpha(t_{0})=i .
	\end{aligned}\right.
\end{eqnarray}
For any $ T > t_{0} $, by applying It$\rm\hat{o}$'s formula to $ \left \langle e^{K\cdot}x_{1}(\cdot),e^{K\cdot}y^{*}(\cdot) \right \rangle$
on the finite
interval $ [t_{0}, T] $ and taking expectation, we arrive at the following expression (where $ s $ is abbreviated, except for the term 
$e^{2Ks}$).
\\$\begin{aligned}
&E\left\langle e^{K T} x_{1}(T), e^{K T} y^{*}(T)\right\rangle\\&=E \int_{t_{0}}^{T} e^{2 K s}\big[\left\langle B_{1}^{T}(\alpha(s)) y^{*}+D_{1}^{T}(\alpha(s)) z^{*}, u_{1}\right\rangle-\left\langle x_{1}, Q(\alpha(s)) x^{*}+S^{T}(\alpha(s)) u^{*}\right\rangle\big] d s.
\end{aligned}$
\\Following from Lemma 3.3, we conclude the following 
\begin{eqnarray}
	\label{tia}\\ \nonumber E \int_{t_{0}}^{+\infty} e^{2 K s}\big[\left\langle B_{1}^{T}(\alpha(s)) y^{*}+D_{1}^{T}(\alpha(s)) z^{*}, u_{1}\right\rangle-\left\langle x_{1}, Q(\alpha(s)) x^{*}+S^{T}(\alpha(s)) u^{*}\right\rangle\big] d s=0.
\end{eqnarray}

Subsequently, we evaluate the following
\\$$\begin{aligned}
&J^{K}(t_{0},x_{t_{0}},u_{1}^{*}(\cdot)+\varepsilon u_{1}(\cdot),u_{2}^{*}(\cdot))-J^{K}(t_{0},x_{t_{0}},u_{1}^{*}(\cdot),u_{2}^{*}(\cdot))\\
&=\frac{\varepsilon}{2} E \int_{t_{0}}^{+\infty} e^{2 K s}\left\langle\Pi(s,\alpha(s))\left(\begin{array}{l}
2x^{*}+\varepsilon x^{*} \\
2u_{1}^{*}+\varepsilon u_{1}^{*} \\ 
2u_{2}^{*}
\end{array}\right),\left(\begin{array}{l}
x \\
u_{1} \\
0
\end{array}\right)\right\rangle d s\\
&=\varepsilon E\int_{t_{0}}^{+\infty} e^{2 K s}\big[\left \langle Q(\alpha(s)x^{*})+S^{T}(\alpha(s))u^{*},x_{1}  \right \rangle \\&\quad + \left \langle S_{1}(\alpha(s))x^{*} +R_{11}(\alpha(s))u_{1}^{*}+
R_{12}(\alpha(s))u_{2}^{*},u_{1}\right \rangle \big]ds\\&\quad+\frac{\varepsilon ^{2}}{2}E \int_{t_{0}}^{+\infty}\big[\left \langle Q(\alpha(s))x_{1},x_{1} \right \rangle+2\left \langle S_{1}(\alpha(s))x_{1},u_{1} \right \rangle +\left \langle R_{11}(\alpha(s))u_{1},u_{1} \right \rangle  \big]ds,
\end{aligned},$$
\\where $\Pi(s,\alpha(s))=\left(\begin{array}{lll}
Q(s,\alpha(s)) & S_{1}(s,\alpha(s))^{T} & S_{2}(s,\alpha(s))^{T} \\
S_{1}(s,\alpha(s)) & R_{11}(s,\alpha(s)) & R_{12}(s,\alpha(s)) \\
S_{2}(s,\alpha(s)) & R_{21}(s,\alpha(s)) & R_{22}(s,\alpha(s))
\end{array}\right)$.
\\The combination of the above equation and (\ref{tia}) leads to 
\begin{eqnarray}
\begin{aligned}
&J^{K}(t_{0},x_{t_{0}},u_{1}^{*}(\cdot)+\varepsilon u_{1}(\cdot),u_{2}^{*}(\cdot))-J^{K}(t_{0},x_{t_{0}},u_{1}^{*}(\cdot),u_{2}^{*}(\cdot))\\
&=\varepsilon E\int_{t_{0}}^{+\infty} e^{2 K s}\big[\left\langle B_{1}^{T}(\alpha(s)) y^{*}+D_{1}^{T}(\alpha(s)) z^{*}+S_{1}(\alpha(s))x^{*},u_{1}\right\rangle \\&\quad+\left\langle R_{11}(\alpha(s))u_{1}^{*}+R_{12}(\alpha(s))u_{2}^{*},u_{1}\right\rangle\big]ds\\
&\quad+\frac{\varepsilon ^{2}}{2}E \int_{t_{0}}^{+\infty}\big[\left \langle Q(\alpha(s))x_{1},x_{1} \right \rangle+2\left \langle S_{1}(\alpha(s))x_{1},u_{1} \right \rangle +\left \langle R_{11}(\alpha(s))u_{1},u_{1} \right \rangle  \big]ds.
\end{aligned}\label{j11}
\end{eqnarray}

Similarly, for $\forall \, u_{2}(\cdot) \in \mathcal{U}_{2}[t_{0}, +\infty),\varepsilon \in \mathbb{R}$,
\begin{eqnarray}
	\begin{aligned}
	&J^{K}(t_{0},x_{t_{0}},u_{1}^{*}(\cdot),u_{2}^{*}(\cdot)+\varepsilon u_{2}(\cdot))-J^{K}(t_{0},x_{t_{0}},u_{1}^{*}(\cdot),u_{2}^{*}(\cdot))\\
	&=\varepsilon E\int_{t_{0}}^{+\infty} e^{2 K s}\big[\left\langle B_{2}^{T}(\alpha(s)) y^{*}+D_{2}^{T}(\alpha(s)) z^{*}+S_{2}(\alpha(s))x^{*},u_{2}\right\rangle \\&\quad+\left\langle R_{22}(\alpha(s))u_{2}^{*}+R_{21}(\alpha(s))u_{1}^{*},u_{2}\right\rangle\big]ds\\
	&\quad+\frac{\varepsilon ^{2}}{2}E \int_{t_{0}}^{+\infty}\big[\left \langle Q(\alpha(s))x_{2},x_{2} \right \rangle+2\left \langle S_{2}(\alpha(s))x_{2},u_{2} \right \rangle +\left \langle R_{22}(\alpha(s))u_{2},u_{2} \right \rangle  \big]ds.
	\end{aligned}\label{j22}
\end{eqnarray}
If (\ref{u**}) is valid and $\varepsilon \in \mathbb{R}$, Assumption D implies that
\\$J^{K}\left(t_{0}, x_{t_{0}} ; u_{1}^{*}(\cdot), u_{2}^{*}(\cdot)\right) \le J^{K}\left(t_{0}, x_{t_{0}} ; u_{1}^{*}(\cdot)+\varepsilon u_{1}(\cdot), u_{2}^{*}(\cdot)\right) \quad \forall \, u_{1}(\cdot) \in \mathcal{U}_{1}[t_{0}, +\infty), $
\\$J^{K}\left(t_{0}, x_{t_{0}} ; u_{1}^{*}(\cdot), u_{2}^{*}(\cdot)\right) \ge J^{K}\left(t_{0}, x_{t_{0}} ; u_{1}^{*}(\cdot), u_{2}^{*}(\cdot)+\varepsilon u_{2}(\cdot)\right) \quad \forall \, u_{2}(\cdot) \in \mathcal{U}_{2}[t_{0}, +\infty).$
\\This means that $u^{*}(\cdot)$ is optimal, and sufficiency has been demonstrated.

For the necessity, let $u^{*}(\cdot)$ be an optimal control. Then for $\forall \, u_{1}(\cdot) \in \mathcal{U}_{1}[t_{0}, +\infty),  \varepsilon \in \mathbb{R}$, (\ref{ju1})  is true.  Accordingly, (\ref{j11}) results in
\begin{eqnarray}
\begin{aligned}
&\varepsilon E\int_{t_{0}}^{+\infty} e^{2 K s}\big[\big\langle B_{1}^{T}(\alpha(s)) y^{*}+D_{1}^{T}(\alpha(s)) z^{*}+S_{1}(\alpha(s))x^{*} \\&\quad \quad+R_{11}(\alpha(s))u_{1}^{*}+R_{12}(\alpha(s))u_{2}^{*},u_{1}\big\rangle\big]ds
+\frac{\varepsilon ^{2}}{2}E \int_{t_{0}}^{+\infty}\big[\left \langle Q(\alpha(s))x_{1},x_{1} \right \rangle\\&\quad \quad+2\left \langle S_{1}(\alpha(s))x_{1},u_{1} \right \rangle  +\left \langle R_{11}(\alpha(s))u_{1},u_{1} \right \rangle  \big]ds\ge 0.
\end{aligned}
\label{j111}
\end{eqnarray}
When $\varepsilon>0$, (\ref{j111}) implies
\begin{eqnarray}
\begin{aligned}
& E\int_{t_{0}}^{+\infty} e^{2 K s}\big[\big\langle B_{1}^{T}(\alpha(s)) y^{*}+D_{1}^{T}(\alpha(s)) z^{*}+S_{1}(\alpha(s))x^{*} \\&\quad \quad+R_{11}(\alpha(s))u_{1}^{*}+R_{12}(\alpha(s))u_{2}^{*},u_{1}\big\rangle\big]ds+\frac{\varepsilon}{2}E \int_{t_{0}}^{+\infty}\big[\left \langle Q(\alpha(s))x_{1},x_{1} \right \rangle\\&\quad \quad+2\left \langle S_{1}(\alpha(s))x_{1},u_{1} \right \rangle +\left \langle R_{11}(\alpha(s))u_{1},u_{1} \right \rangle  \big]ds\ge 0.
\end{aligned}
\nonumber
\end{eqnarray}
Taking the limit as $\varepsilon \to 0^{+}$, it shows 
\begin{eqnarray}
\begin{aligned}
&E\int_{t_{0}}^{+\infty} e^{2 K s}\big[\big\langle B_{1}^{T}(\alpha(s)) y^{*}+D_{1}^{T}(\alpha(s)) z^{*}+S_{1}(\alpha(s))x^{*} \\&\quad \quad \quad\quad \quad+R_{11}(\alpha(s))u_{1}^{*}+R_{12}(\alpha(s))u_{2}^{*},u_{1}\big \rangle\big]ds\ge 0.
\end{aligned}\label{su1}
\end{eqnarray}
If $\varepsilon<0$ and $\varepsilon \to 0^{-}$, this leads to 
\begin{eqnarray}
\begin{aligned}
&E\int_{t_{0}}^{+\infty} e^{2 K s}\big[\big\langle B_{1}^{T}(\alpha(s)) y^{*}+D_{1}^{T}(\alpha(s)) z^{*}+S_{1}(\alpha(s))x^{*} \\&\quad \quad \quad\quad \quad+R_{11}(\alpha(s))u_{1}^{*}+R_{12}(\alpha(s))u_{2}^{*},u_{1}\big\rangle\big]ds\le 0.
\end{aligned} \label{su2}
\end{eqnarray}
With the help of (\ref{su1}) and (\ref{su2}), it is evident that 
\begin{eqnarray}
\begin{aligned}
&E\int_{t_{0}}^{+\infty} e^{2 K s}\big[\big\langle B_{1}^{T}(\alpha(s)) y^{*}+D_{1}^{T}(\alpha(s)) z^{*}+S_{1}(\alpha(s))x^{*} \\&\quad \quad \quad\quad \quad+R_{11}(\alpha(s))u_{1}^{*}+R_{12}(\alpha(s))u_{2}^{*},u_{1}\big \langle\big]ds=0.
\end{aligned}\label{3} 
\end{eqnarray}

Analogously, for $\forall \, u_{2}(\cdot) \in \mathcal{U}_{2}[t_{0}, +\infty),  \varepsilon \in \mathbb{R}$, we have 
\begin{eqnarray}
\begin{aligned}
&E\int_{t_{0}}^{+\infty} e^{2 K s}\big[\big\langle B_{2}^{T}(\alpha(s)) y^{*}+D_{2}^{T}(\alpha(s)) z^{*}+S_{2}(\alpha(s))x^{*} \\&\quad \quad \quad\quad \quad+R_{22}(\alpha(s))u_{2}^{*}+R_{21}(\alpha(s))u_{1}^{*},u_{2}\big \rangle\big]ds=0.
\end{aligned}\label{4}
\end{eqnarray}
In view of  the arbitrariness of $ u_{1}(\cdot ) $ and $ u_{2}(\cdot ) $, the (\ref{3}) and (\ref{4}) imply (\ref{u**}), with the necessity being proven.
\end{proof}

\textit{Remark}: Since this paper focuses on the two-person zero-sum optimal control problem, Assumption D must be applied to the coefficients before different controls are introduced. In most prior studies, the cost function in (\ref{jj}) assumes $ K=0 $; however, we consider the more general case. To address the challenges arising during the derivation, the adjoint equation (\ref{ady}) has been modified accordingly, distinguishing it from the equations in previous works.

In accordance with Assumption D, $u^{*}$ is represented as
\begin{eqnarray}
	u^{*}(\cdot)=-R(\cdot)^{-1}\big[B(\cdot)^{T} y^{*}(\cdot)+D(\cdot)^{T} z^{*}(\cdot)+S(\cdot)^{T}x^{*}(\cdot)\big].
\end{eqnarray}
Upon substituting $u^{*}$ into (\ref{xx}) and (\ref{ady}), the following coupled equations arise:
\begin{eqnarray}
	\left\{\begin{aligned}
	d x^{*}(s)= & \big[\left(A(s,\alpha(s))-B(s,\alpha(s)) R^{-1}(s,\alpha(s)) S(s,\alpha(s))\right) x^{*}(s)\\&-B(s,\alpha(s)) R^{-1}(s,\alpha(s)) B^{T}(s,\alpha(s)) y^{*}(s)\\&-B(s,\alpha(s)) R^{-1}(s,\alpha(s)) D^{T}(s,\alpha(s)) z^{*}(s)\big] d s+H(s)dB^{H}(s) \\
	& +\sum_{i=1}^{d}\big[\left(C_{i}(s,\alpha(s))-D_{i}(s,\alpha(s)) R^{-1}(s,\alpha(s)) S(s,\alpha(s))\right) x^{*}(s)\\&-D_{i}(s,\alpha(s)) R^{-1}(s,\alpha(s)) B^{T}(s,\alpha(s)) y^{*}\\&-D_{i}(s,\alpha(s)) R^{-1}(s,\alpha(s)) D^{T}(s,\alpha(s)) z^{*}(s)\big] d W_{i}(s),  \\
	d y^{*}(s)=& -\big[(2 K I+A(s,\alpha(s))\\&-B(s,\alpha(s)) R^{-1}(s,\alpha(s)) S(s,\alpha(s)))^{T} y^{*}(s)\\&+\left(C(s,\alpha(s))-D(s,\alpha(s)) R^{-1}(s,\alpha(s)) S(s,\alpha(s))\right)^{T} z^{*}(s)\\&+\left(Q(s,\alpha(s))-S^{T}(s,\alpha(s)) R^{-1}(s,\alpha(s)) S(s,\alpha(s))\right) x^{*}(s)\big] d s \\
	& +\sum_{i=1}^{d} z_{i}^{*}(s) d W_{i}(s)+r^{*}(s)dB^{H}(s)+f^{*}(s)\cdot d\mathcal{M}(S),   \\
	x^{*}(t_{0})=  &x_{t_{0}},\quad \alpha(t_{0})=i .
	\end{aligned}\right.\label{onlyx}
\end{eqnarray}
\textbf{Example:} Let $n=d=1$. Consider the following equation:
\begin{eqnarray}
	\begin{cases}
	dx(s)=\left[u_{1}^{*}(s)+u_{2}^{*}(s) \right]ds+\left[u_{1}^{*}(s)+u_{2}^{*}(s) \right]dW(s)+dB^{H}(s) ,\\
	x(0)=x,
	\end{cases}\label{examplex}
\end{eqnarray}
with the cost function expressed as
$$J(0,x;u_{1},u_{2})=\frac{1}{2}E\int_{0}^{+\infty}\left[u_{1}^{2}(s)-u_{2}^{2}(s)\right]ds.$$
\\As suggested by equation (\ref{examplex}), we present the following BSDE 
\begin{eqnarray}
	dy(s)=z(s^{*})dW(s)+dB^{H}(s)+d\mathcal{M}(s).\label{exampley}
\end{eqnarray}
 Then, the equation (\ref{u**}) is re-expressed in the following manner:
\begin{equation}
	\begin{cases}
	  y^{*}(s) +z^{*}(s)+u_{1}^{*}(s)=0,\quad a.s., \\
	  y^{*}(s) +z^{*}(s)-u_{2}^{*}(s)=0,\quad a.s.,
	\end{cases}\label{examplecon}
\end{equation}
where $s\in [0,+\infty)$. According to the proof of Theorem 6.1, if equations (\ref{examplex}) and (\ref{exampley}) satisfy condition (\ref{examplecon}), then conditions (\ref{ju1}) and (\ref{ju2}) hold true based on (\ref{j11}) and (\ref{j22}). Conversely, if conditions (\ref{ju1}) and (\ref{ju2}) hold, it follows that the equation satisfies condition (\ref{examplecon}). Therefore, $u^{*}(\cdot)=(u_{1}^{*}(\cdot)^{T},u_{2}^{*}(\cdot)^{T})^{T}$ is an optimal control.

From the cost function (\ref{jj}), it is clear that 
$ x $ and $ u $ are associated with a coefficient 
$S(\cdot)$, which governs the cross-term. In the context of this paper, the presence of 
$S(\cdot)$ might influence the derivation procedure. Thus, we examine the effect of 
$S(\cdot)$ on the stochastic differential equation in two cases: when $S(\cdot)$ is zero and when it is non-zero.
\subsection{The coefficient of the cross term $ S (\cdot ) = 0 $}
When $ S (\cdot ) = 0 $, the FBSDE (\ref{onlyx}) takes the form
\begin{eqnarray}
\left\{\begin{aligned}
d x^{*}&(s)=  \big[A(s,\alpha(s)) x^{*}-B(s,\alpha(s)) R^{-1}(s,\alpha(s)) B^{T}(s,\alpha(s)) y^{*}(s)\\&-B(s,\alpha(s)) R^{-1}(s,\alpha(s)) D^{T}(s,\alpha(s)) z^{*}(s)\big] d s +H(s)dB^{H}(s) \\
& +\sum_{i=1}^{d}\big[C_{i}(s,\alpha(s))x^{*}(s)-D_{i}(s,\alpha(s)) R^{-1}(s,\alpha(s)) B^{T}(s,\alpha(s)) y^{*}(s)\\&-D_{i}(s,\alpha(s)) R^{-1}(s,\alpha(s)) D^{T}(s,\alpha(s)) z^{*}(s)\big] d W_{i}(s), \quad s \in[t_{0}, +\infty), \\
d y^{*}&(s)=  -\big[(2 K I+A(s,\alpha(s)))^{T} y^{*}(s)+ C^{T}(s,\alpha(s)) z^{*}(s)\\&+Q(s,\alpha(s)) x^{*}(s)\big] d s  +\sum_{i=1}^{d} z_{i}^{*}(s) d W_{i}(s)\\
&+r^{*}(s)dB^{H}(s)+f^{*}(s)\cdot d\mathcal{M}(s), \quad s \in[t_{0}, +\infty), \\
x^{*}&(t_{0})=  x_{t_{0}},\quad\alpha(t_{0})=i .
\end{aligned}\right.\label{s0}
\end{eqnarray}
With $ S (\cdot ) = 0  $, which is a special case of equation (\ref{onlyx}), the assumptions remain unchanged. Therefore, equation (\ref{s0}) also admits a unique solution by applying the same technique.
\subsection{The coefficient of the cross term $ S (\cdot ) \ne 0 $}
As explained in \cite{huang2014solvability}, a linear transformation enables the reduction of the general case 
$ S(\cdot) \ne 0 $ to the special case $ S(\cdot) = 0 $. Let
\begin{eqnarray}
	\widetilde{u}(\cdot)=u(\cdot)+R(\cdot)^{-1} S(\cdot) x(\cdot).
\end{eqnarray}
Hence, equations (\ref{xx}) and (\ref{jj}) take the following form
\begin{eqnarray}
	\left\{\begin{aligned}
	&d x(s)=\left[\widetilde{A}(s,\alpha(s)) x(s)+B(s,\alpha(s)) \widetilde{u}(s)\right] d s+H(s)dB^{H}(s)\\&\quad  \quad \quad+\sum_{i=1}^{d}\left[\widetilde{C}_{i}(s,\alpha(s)) x(s)+D_{i}(s,\alpha(s)) \widetilde{u}(s)\right] d W_{i}(s),\quad s \in[t_{0}, +\infty) \\
	&x(t_{0})=x_{t_{0}},\quad \alpha(t_{0})=i,
	\end{aligned}\right.\label{newx}
\end{eqnarray}
and
\begin{eqnarray}
	\begin{aligned}
	&\widetilde{J}^{K}\left(t_{0}, x_{t_{0}} ; \widetilde{u}(\cdot)\right):=J^{K}\left(t_{0}, x_{t_{0}} ; u(\cdot)\right) \\&
	=\frac{1}{2} E \int_{t_{0}}^{+\infty} e^{2 K s}\left[\langle\widetilde{Q}(s,\alpha(s)) x(s), x(s)\rangle+\langle R(s,\alpha(s)) \widetilde{u}(s), \widetilde{u}(s)\rangle\right] \mathrm{d} s,
	\end{aligned}\label{newJ}
\end{eqnarray}
where $\widetilde{A}(\cdot)  :=A(\cdot)-B(\cdot) R(\cdot)^{-1} S(\cdot),  \widetilde{C}(\cdot):=C(\cdot)-D(\cdot) R(\cdot)^{-1} S(\cdot) ,
\widetilde{Q}(\cdot) :=Q(\cdot)-S(\cdot)^{T} R(\cdot)^{-1} S(\cdot)$.

The cross term disappears upon replacing (\ref{xx}) and (\ref{jj}) with (\ref{newx}) and (\ref{newJ}), enabling a simplification to a special case for further examination.

\section{Conclusion}
This paper investigates a class of two-person zero-sum stochastic differential equations influenced by Markov chains and fractional Brownian motion over an infinite time horizon. These models are especially relevant in contexts where systems are influenced by both random jumps, modeled by Markov chains, and long-range dependencies captured by fractional Brownian motion. We introduce a novel It$\rm\hat{o}$'s formula for the forward-backward stochastic differential equations governing the system's dynamics. By leveraging this formula, we transform the system's evolution into a form amenable to precise analysis and solution. The existence and uniqueness of solutions to these equations are rigorously established. We then examine the optimal control problem in the context of the two-person zero-sum game, deriving explicit optimal control strategies for both players. Additionally, we explore how the dynamics of the system are influenced by the presence or absence of the cross term 
$ S(\cdot) $. Our results contribute to the theory of stochastic control and differential games by offering a comprehensive framework that combines Markov and fBm components. These findings have potential applications in fields like financial engineering, dynamic pricing, and optimal resource management, where hybrid models provide a more accurate representation of real-world systems. Future work will extend the model by incorporating time-delay effects to investigate their impact on the system’s behavior.

%\appendix
%\section{Summary of major changes}
\bibliography{bibtex}

\end{document}